\documentclass[a4paper]{amsart}

\usepackage{cite}
\usepackage{amsmath}
\usepackage{amssymb}
\usepackage{amsthm}   %-------proof
\usepackage{empheq}   %-------to use \empheq
\usepackage{graphicx}
\usepackage{color}
\usepackage{appendix}
\usepackage{hyperref}
\usepackage[all]{xy}
\usepackage{mathrsfs}

\newtheorem{theorem}{Theorem}[section]
\newtheorem{lemma}[theorem]{Lemma}
\newtheorem{proposition}[theorem]{Proposition}
\newtheorem{corollary}[theorem]{Corollary}

\theoremstyle{definition}
\newtheorem{definition}[theorem]{Definition}
\newtheorem{example}[theorem]{Example}

\theoremstyle{remark}
\newtheorem{remark}[theorem]{Remark}

\numberwithin{equation}{section}
\def\EE{{\mathscr{E}}}
\def\FF{{\mathscr{F}}}
\def\CC{{\mathscr{C}}}
\def\bs{{\mathtt{s}}}

\def\ss{{\mathtt{s}}}
\def\SS{{\mathbf{S}}}
\def\ff{{\mathfrak{f}}}
\def\s{{\mathfrak{s}}}
\def\e{{\mathrm{e}}}
\def\t{{\mathrm{t}}}

\begin{document}

\title[Effective intervals and D-Subspaces]{Effective intervals %of 1-dim symmetric diffusion
and regular Dirichlet subspaces}

\author{ Liping Li}
\address{RCSDS, HCMS, Academy of Mathematics and Systems Science, Chinese Academy of Sciences, Beijing 100190, China.}
%\curraddr{Department of Mathematics and Statistics,
%Case Western Reserve University, Cleveland, Ohio 43403}
\email{liliping@amss.ac.cn}
\thanks{The first named author is partially supported by a joint grant (No. 2015LH0043) of China Postdoctoral Science Foundation and Chinese Academy of Science, China Postdoctoral Science Foundation (No. 2016M590145), NSFC (No. 11688101) and Key Laboratory of Random Complex Structures and Data Science, Academy of Mathematics and Systems Science, Chinese Academy of Sciences (No. 2008DP173182). The third named author is partially supported by NSFC No. 11271240.}

\author{Wenjie Sun}
\address{Shanghai Center for Mathematical Sciences, Fudan University, Shanghai 200433, China.}
\email{wjsun14@fudan.edu.cn}

\author{Jiangang Ying}
\address{School of Mathematical Sciences, Fudan University, Shanghai 200433, China.}
\email{jgying@fudan.edu.cn}

%    General info
%\subjclass[2000]{Primary 37K55, 37J40; Secondary 35B35, 35Q35}
%\date{January 1, 2001 and, in revised form, June 22, 2001.}



\begin{abstract}
%Given two regular Dirichlet forms $(\EE^1,\FF^1)$ and $(\EE^2,\FF^2)$ on the same $L^2$-space, $(\EE^1,\FF^1)$ is called a regular Dirichlet subspace of $(\EE^2,\FF^2)$ if $\FF^1\subset \FF^2$ and $\EE^2(u,v)=\EE^1(u,v)$ for any $u,v\in \FF^1$.
It is shown in \cite{LY17} that a regular and local Dirichlet form on an interval can be % obtained in  \cite{LY17}.
represented by so-called effective intervals with scale functions.
This paper focuses on how to operate on effective intervals to obtain regular Dirichlet subspaces.% are related to effective intervals.
The first result is a complete characterization for a Dirichlet form to be a regular subspace of such a Dirichlet form in terms of effective intervals. Then we give an explicit road map how to obtain all regular Dirichlet subspaces from a local and regular Dirichlet form on an interval, by a series of intuitive operations on the effective intervals in the representation above. Finally applying previous results, we shall prove that every regular and local Dirichlet form has a special standard core generated by a continuous and strictly increasing function.
\end{abstract}

\maketitle

\tableofcontents

\section{Introduction}
A Dirichlet form is a closed and symmetric bilinear form with Markovian property on $L^2(E, m)$ space, where $E$ is a nice topological space and $m$ is a fully supported Radon measure on $E$. Due to a series of important works by M. Fukushima, M. L. Silverstein in 1970's,  a regular Dirichlet form is always associated with a symmetric Markov process
uniquely by the transition semigroup, so we do not distinguish them for convenience. For example, when we say a subspace of a symmetric diffusion, it means a subspace of its associated Dirichlet form. We refer the notions and terminologies in the theory of Dirichlet forms to \cite{CF12, FOT11}.

The notion of regular Dirichlet subspace of a Dirichlet form was first raised by the third named author and his co-authors in \cite{FFY05}. The dual notion, regular Dirichlet extension, was raised in \cite{LY16} by the first and third named authors together. These two notions are about the inclusion relation between two Dirichlet spaces. Namely, let $E$ be a locally compact separable metric space and $m$ a fully supported Radon measure on $E$. Given two regular Dirichlet forms $(\EE^1,\FF^1)$ and $(\EE^2,\FF^2)$ on the same Hilbert space $L^2(E,m)$, if
\[\FF^1\subset\FF^2,\quad \EE^2(u,v)=\EE^1(u,v),\quad \forall ~ u,v\in \FF^1,\]
we say $(\EE^1,\FF^1)$ is a \emph{regular Dirichlet subspace} or simply a \emph{D-subspace} of $(\EE^2,\FF^2)$, and  conversely, $(\EE^2,\FF^2)$ is a \emph{regular Dirichlet extension} or simply \emph{D-extension} of $(\EE^1,\FF^1)$.

The Brownian motion is a classical and fundamental model in the theory of stochastic processes. It is well known that the Dirichlet form associated with 1-dim (an abbreviation for one-dimensional) Brownian motion is $(\frac{1}{2}\mathbf{D},H^1(\mathbb{R}))$, where $H^1(\mathbb{R})$ is the $1$-Sobolev space and $\mathbf{D}$ is the Dirichlet integral, i.e., for any $u,v\in H^1(\mathbb{R})$,
\[ \mathbf{D}(u,v)=\int_{\mathbb{R}}u'(x)v'(x)dx. \]
The D-subspaces and D-extensions of 1-dim Brownian motion have been studied in \cite{FFY05} and \cite{LY16} respectively. It is shown that any D-subspace of $(\frac{1}{2}\mathbf{D},H^1(\mathbb{R}))$ corresponds to an irreducible (or `regular') symmetric diffusion process on $\mathbb{R}$ in the sense that $\mathbf{P}_x(\sigma_y<\infty)>0$ for any $x,y\in \mathbb{R}$, where $\sigma_y$ is the hitting time of $\{y\}$ relative to this diffusion. Moreover, such a subspace may be characterized uniquely by a so-called scale function $\bs$ (Cf. \cite[V.46]{RW87}) satisfying that $\bs$ is absolutely continuous and
\begin{equation}\label{EQ1SAE}
	\bs'=0 \text{ or } 1 \text{ a.e.}
\end{equation}
However, the D-extension of  $(\frac{1}{2}\mathbf{D},H^1(\mathbb{R}))$ is not necessarily irreducible. In other words, it admits non-trivial invariant components. It is shown as the main result of \cite{LY16} that the state space $\mathbb{R}$ of each D-extension of $(\frac{1}{2}\mathbf{D},H^1(\mathbb{R}))$ may be essentially decomposed into at most countable invariant intervals and an exceptional set, and on each interval, it behaves as an irreducible diffusion characterized by some appropriate scale function. We refer further explorations about D-subspaces of some other Dirichlet forms to \cite{FHY10, LY15, LY16-2, LY17-2} and \cite{LY14}.

In this paper, we shall essentially focus on regular and strongly local Dirichlet forms. The state space is the real line $\mathbb{R}$ if not otherwise stated. The results may be generalized to regular and local Dirichlet forms on an interval without real difficulty as indicated in the last section.
Since the unique probabilistic counterpart of such a form is  %associated uniquely with  %which are in one-to-one correspondence with
a 1-dim symmetric diffusion or a symmetric diffusion on $\mathbb{R}$, we often abuse these two notions for simplicity and intuition, if no confusion will be caused.

Inspired by the work on regular Dirichlet extensions of 1-dim Brownian motion, the representation of Dirichlet forms associated with 1-dim symmetric diffusions, including non-irreducible ones, was studied in \cite{LY17}.
The main result in \cite{LY17} will be reviewed in \S\ref{SEC2}.  Roughly speaking, in spite of the possible killing insides, such 1-dim symmetric diffusion lives on at most countable disjoint intervals, called effective intervals there, and every point outside these intervals (probably non-trivial) is a `trap' of the diffusion in the sense that all the trajectories starting from this point will never leave. On each effective interval, it is an irreducible diffusion characterized by an `adapted' scale function. Thus the associated Dirichlet form is described in unique way by a class of at most countable pairs $\{(I_n, \bs_n): n\geq 1\}$, where $I_n$ is the effective interval and $\bs_n$ is an `adapted' scale function on it.

The set of effective intervals with adapted scale functions is a probabilistic point of view to look at a Dirichlet form, which is purely an analytic object. It should be true that any property of such a Dirichlet form may be characterized intuitively by its effective intervals. The main purpose of this paper is to characterize the relation of regular Dirichlet subspace/extension through effective intervals and make this analytic notion more intuitive.

Three main results will be presented in this paper. The first result, stated as Theorem~\ref{thm2.1}, is a necessary and sufficient condition for one 1-dim symmetric diffusion to be a D-subspace of another in terms of their effective intervals.  % $\{(I_n,\bs_n):n\ge 1\}$ and $\{(\mathtt{I}_k, \mathfrak{s}_k): k\geq 1\}$ respectively,
 %This representation theorem leads us to study how to get the regular Dirichlet subspaces or extensions of symmetric linear diffusions through effective intervals. This is also the main purpose of this paper. As indicated in \cite[Theorem~2.1]{LY15}, any regular Dirichlet subspace (or extension) of a local Dirichlet form is still local. Applying the above representation theorem to this subspace, we obtain that it is characterized by another class of effective intervals  %at most countable pairs
 %$\{(\mathtt{I}_k, \mathfrak{s}_k): k\geq 1\}$.  %where $\mathtt{I}_k$ is the new effective interval and $\mathfrak{s}_k$ is the scale function on it.
 %Thus it is interesting to find the conditions between these two classes $\{(I_n,\bs_n): n\geq 1\}$ and $\{(\mathtt{I}_k, \mathfrak{s}_k): k\geq 1\}$, which could characterize the relation of regular subspaces. These conditions are formulated in Theorem~\ref{thm2.1}.
Inspired by the result for 1-dim Brownian motion where the scale function with \eqref{EQ1SAE} plays an essential role, we introduce a new conception, named the scale measure, which is the sum of all measures induced by scale functions on effective intervals.  %The first main result, Theorem~\ref{thm2.1}, gives a necessary and sufficient condition in terms of effective intervals for one of two linear diffusions to be a regular Dirichlet subspace of another, where
The condition \eqref{EQ3DLS} in Theorem~\ref{thm2.1} is similar to \eqref{EQ1SAE}.  % if one is a regular Dirichlet subspace of another.
Particularly, \eqref{EQ3DLS} coincides with \eqref{EQ1SAE} when returning to 1-dim Brownian motion.
Nevertheless, Theorem~\ref{thm2.1} gives only a criterion that a Dirichlet form is a D-subspace of the other, and it is more interesting to know whether it is possible to obtain a D-subspace through some operation on the effective intervals. The second result is to answer this question positively and draw a concrete road map to illustrate how to do this.  %obtain all regular Dirichlet subspaces of a local Dirichler form, we classify them according to their scale measures.
Roughly speaking, the condition \eqref{EQ3DLS} may be viewed as an operation which multiplies a factor to its scale measure and gives a new scale measure. However different D-subspaces may share the same scale measure. Therefore  %we need another operation on effective intervals
once identifying the scale measure, we need `interval-merge' operations to obtain all D-subspaces taking this scale measure,  %Intuitively speaking, these operations are proceeded to merging t
 which means that the original effective intervals are firstly grouped and then merged according to some rule into new ones. Several examples are also raised to illustrate these operations. The idea of `interval-merge' was originated in \cite[\S3.5]{LY17}, where it was used to identify the closure of $C_c^\infty(\mathbb{R})$ in a 1-dim symmetric diffusion. In some sense, the discussion in \cite[\S3]{LY17} could be treated as a special case of what we shall consider here. Intuitively speaking, 1-dim diffusion may be viewed as an electron wandering on an electric network.
The operation of scale-shrink is to reduce the resistance of network by placing super-conductance and the operation of interval-merge is to connect some broken networks, which can be merged, together. 
The third result, as an application of the second one, is to study a special class of D-subspaces generated by
\[
	\CC_\ff:=C_c^\infty\circ \ff=\{\varphi\circ \ff: \varphi\in C_c^\infty(\ff(\mathbb{R}))\},
\]
where $\ff$ is a continuous and tightly increasing function on $\mathbb{R}$. We find that the scale measure of this D-subspace is the absolutely continuous part of original scale measure with respect to $d\ff$, and the optional interval-merge to attain this D-subspace is performed on equivalence classes obtained by the so-called $\ff$-scale-connection in Definition~\ref{DEF53}. This can be used to prove an interesting and useful fact that every regular and strongly local Dirichlet form on $L^2(\mathbb{R},m)$ has a special standard core generated this way, just as Brownian motion has a special standard core consisting of smooth functions.

{As a dual conception, D-extensions enjoy the same characterization result as D-subspaces. Particularly, Theorem~\ref{thm2.1} also characterizes D-extensions of a 1-dim diffusion completely. We left further discussions about D-extensions in a future study.}%Certainly, the conditions in it should be changed into those based on D-extensions. That means, the effective intervals of a D-extension are `finer' (see Remark~\ref{RM32}~(1)) than original effective intervals, and their scale measure is a certain `enlargement' of original one.

  %We try to focus on two problems in this paper. One is a criterion for subspaces or extensions given two regular and strongly local Dirichlet forms, the other one is how to attain all the subspaces and extensions from one such form.

This paper is organized as follows. In \S\ref{SEC2}, we shall briefly review the representation theorem (Cf. \cite{LY17}) for regular and strongly local Dirichlet forms. In \S\ref{SEC3}, a complete characterization of a D-subspace (or D-extension) for a 1-dim symmetric diffusion is given in terms of effective intervals and scale measures.  %The scale measure can be viewed as a summation of the Lebesgue-Stieltjes measures induced by the scale functions.
It turns out that if one regular and strongly local Dirichlet form is a D-subspace of another, then the effective intervals have to be `coarser', and the scale measure has to be reduced in the way presented by \eqref{EQ3DLS}.
The section \S\ref{SEC4} is devoted to draw a road map from the original Dirichlet form to its D-subspaces.  %There are two things we need to deal with, the effective intervals and the scale functions.
We shall introduce two kinds of operations. One is called the `scale-shrink' operation, which essentially identifies the scale measure of a D-subspace. The other is called `optional interval-merge' operation, which groups and merges the effective intervals into new ones. Then in Theorem~\ref{THM419}, we shall illustrate that every D-subspace is obtained by firstly a scale-shrink operation and then an optional interval-merge operation. The section \S\ref{SEC5} is an application of this road map. It concerns the D-subspaces generated by a special class of functions. The principal theorem, i.e. Theorem~\ref{THM56}, presents the scale measures and optional interval-merge to attain these D-subspaces. Particularly, a corollary of this result also provides an effective method to find a `nice' special standard core of the Dirichlet form represented in Theorem~\ref{THM21}. Some interesting examples are raised to realize this method. We then prove in Theorem~\ref{T58SSC} that any regular and local Dirichlet form
has a special standard core of this form.
Finally, several further remarks are given in \S\ref{SEC6}. The first one deals with the case that the state space is just an interval. It makes no big difference, but special attentions are needed when we come to the boundaries of the interval. The second remark concerns the killing insides. By using the resurrected transform and killing transform, we can  easily deduce that the only additional condition is that the two Dirichlet forms share the same killing measure.  %The D-extensions of 1-dim symmetric diffusion, a dual notion to D-subspaces.

\subsection*{Notations}
Let us put some often used notations here for handy reference, though we may restate their definitions when they appear.

For $a<b$, $\langle a, b\rangle$ is an interval where $a$ or $b$ may or may not be contained in $\langle a, b\rangle$.
The restrictions of a measure $\mu$ and a function $f$ to an interval $J$ are denoted by $\mu|_J$ and $f|_J$ respectively.
The notation `$:=$' is read as `to be defined as'.
For a scale function $\bs$ (i.e. a continuous and tightly increasing function) on some interval $J$, $d\bs$ represents its associated measure on $J$. Set $\bs(J):=\{\bs(x):x\in J\}$. For two measures $\mu$ and $\nu$, $\mu\ll \nu$ means $\mu$ is absolutely continuous with respect to $\nu$, and $\mu\simeq\nu$ means that they are mutually absolutely continuous (or simply equivalent). Given a scale function $\bs$ on $J$ and another function $f$ on $J$, $f\ll \bs$ means $f=g\circ \bs$ for an absolutely continuous function $g$ and $$\frac{df}{d\bs}:=g'\circ \bs.$$ The classes $C_c(J), C^1_c(J)$ and $C^\infty_c(J)$ denote the spaces of all continuous functions on $J$ with compact support, all continuously differentiable functions with compact support and all infinitely differentiable functions with compact support, respectively.

Fix a Markov process $X=(X_t)_{t\geq 0}$ associated with a Dirichlet form $(\EE,\FF)$ on $L^2(E,m)$. If $U$ is an open subset of $E$, then the part Dirichlet form of $(\EE,\FF)$ on $U$ is denoted by $(\EE_U,\FF_U)$ and the part process of $X$ on $U$ is denoted by $X_U$.
All the terminologies about Dirichlet forms are standard and we refer them to \cite{FOT11, CF12}.

\section{A review of 1-dim symmetric diffusions}\label{SEC2}

This section is devoted to a brief review of the representation of regular Dirichlet forms associated with 1-dim symmetric diffusions. Intuitively speaking, a 1-dim symmetric diffusion lives on at most countable disjoint intervals, on each interval it is a `regular' diffusion (for regularity of a 1-dim diffusion, see \cite[\S45]{RW87}) and outside these intervals the diffusion will never move.  These are presented in \cite{LY17} and for readers' convenience, we summarize the main results as follows.

Let $\mathbb{R}$ be the real line, and $m$ a fully supported Radon measure on $\mathbb{R}$. Further let $J:=\langle a, b\rangle$ be an interval, where $a$ or $b$ may or may not be contained in $J$. Note that a `regular' 1-dim diffusion on $J$  is characterized completely by a scale function (uniquely up to a constant), a speed measure and a killing measure.  Take a fixed point in the interior of $J$ as follows
 %\begin{displaymath}
\begin{equation}\label{midpoint}
e:=\left\{
\begin{array}{ll}
\dfrac{a+b}{2}, & |a|+|b|<\infty, \\
a+1, & a>-\infty, b=\infty,\\
b-1, & a=-\infty, b<\infty,\\
0, & a=-\infty, b=\infty.
\end{array}\right.
\end{equation}
 %\end{displaymath}
 %$e$ is served as a middle point of $I$.
Then the family of scale functions on $J$ is given by
\[
\mathbf{S}(J):= \{ \mathtt{s}: J\rightarrow \mathbb{R}: \bs \ \mbox{is strictly increasing and continuous},\ \bs(e)=0
\},
\]
where we impose $\bs(e)=0$ to guarantee the uniqueness of scale function for regular 1-dim diffusion.
Since $\mathtt{s}(x)$ is monotone, we set
\[
\mathtt{s}(a):=\lim_{x\downarrow a}\bs(x),\quad \mathtt{s}(b):= \lim_{x\uparrow b}\bs(x).
	\]
In \cite{LY17}, a scale function $\bs$ is required to be adapted to the interval $J$ in the sense that
\begin{itemize}
\item[$(\textbf{A}_R)$] $a+\bs(a)>-\infty$ if and only if $a\in J$;
\item[$(\textbf{B}_R)$] $b+\bs(b)<\infty$ if and only if $b\in J$.
\end{itemize}
Thus we define a sub-family of scale functions as
$$\mathbf{S}_{\infty}(J):=\{\mathtt{s}\in \mathbf{S}(J): \mathtt{s}\text{ satisfies }(\textbf{A}_R) \text{ and }(\textbf{B}_R)\}.$$
In other words, an open endpoint is unapproachable and a closed endpoint is reflecting for the diffusion on $J$ with a scale function in $\mathbf{S}_\infty(J)$.
On the other hand, the possible absorbing property at two infinities is also needed to be considered:
\begin{itemize}
\item[($\text{L}_R$)] $a=-\infty$, $\bs(-\infty)>-\infty$ and $m((-\infty, 0])<\infty$;
\item[($\text{R}_R$)] $b=\infty$, $\bs(\infty)<\infty$ and $m([0,\infty))<\infty$.
\end{itemize}
For a function $\bs\in \mathbf{S}_\infty(J)$, the Dirichlet form on $L^2(J,m|_J)$ defined by
\begin{equation}\label{EQ2FSU2}
\begin{aligned}
&\FF^{(\bs)}:=\bigg\{u\in L^2(J,m|_J): u\ll \bs, \; \frac{du}{d\bs}\in L^2(J,d\bs); \\
 &\qquad \qquad\qquad  u(a)=0\text{ (resp. }u(b)=0\text{) whenever (}\text{L}_R\text{) (resp. (}\text{R}_R\text{))} \bigg\},  \\
& \EE^{(\bs)}(u,v):=\frac{1}{2}\int_J \frac{du}{d\bs}\frac{dv}{d\bs}d\bs,\quad u,v\in \FF^{(\bs)}
\end{aligned}
\end{equation}
is regular and associated with $m|_J$-symmetric `regular' diffusion on $J$ with scale function $\bs$ (see \cite{FHY10}).

The following theorem is taken from \cite[Corollary~2.13 and Theorem~4.1]{LY17}, which presents a complete representation of regular and strongly local Dirichlet forms on $L^2(\mathbb{R},m)$. Note that the strong local property of Dirichlet form implies that the associated Markov process is continuous, and has no killing inside.

\begin{theorem}\label{THM21}
Let $m$ be a fully supported Radon measure on $\mathbb{R}$. Then $(\EE, \FF)$ is a regular and strongly local Dirichlet form on $L^2(\mathbb{R},m)$ if and only if there exist a set of at most countable disjoint intervals $\{I_n=\langle a_n,b_n\rangle: I_n\subset \mathbb{R}, n\geq 1\}$ with a scale function $\ss_n\in \SS_\infty(I_n)$ for each $n\geq 1$  %and a Radon measure $k$ on $\mathbb{R}$ with $k\ll m$ on $\left(\bigcup_{n\geq 1}I_n \right)^c$
such that
\begin{equation}\label{EQ2FULR}
\begin{aligned}
	&\FF=\left\{u\in L^2(\mathbb{R}, m): u|_{I_n}\in \FF^{(\ss_n)}, \sum_{n\geq 1}\EE^{(\ss_n)}(u|_{I_n}, u|_{I_n})<\infty  \right\},  \\
	&\EE(u,v)=\sum_{n\geq 1}\EE^{(\ss_n)}(u|_{I_n}, v|_{I_n}),\quad u,v \in \FF,
\end{aligned}
\end{equation}
where for each $n\geq 1$, $(\EE^{(\ss_n)}, \FF^{(\ss_n)})$ is given by \eqref{EQ2FSU2} with the scale function $\ss_n$ on $I_n$. Moreover, the intervals $\{I_n: n\geq 1\}$ and scale functions $\{\ss_n:n\geq 1\}$ are uniquely determined, if the difference of order is ignored.
\end{theorem}
\begin{remark}\label{RM22}
This representation theorem is valid a bit more generally.  %for the state space $\mathbb{R}$ but also for any interval.
We refer the general version of Theorem~\ref{THM21} to \cite[Theorem~4.1]{LY17}. Note that the consideration on $\mathbb{R}$ do not lose generality and the reason is presented in \cite[\S2.4]{LY17}.
Our results in general case will be briefly stated in \S\ref{SEC6}.
\end{remark}

Let us give more explanations for the theorem above.  We denote the associated diffusion process of $(\EE,\FF)$ by $(X_t, \mathbf{P}_x)$.  % The measure $k$ is called the killing measure of $(\EE,\FF)$. When $k=0$, $(\EE,\FF)$ is strongly local and the associated diffusion has no killing inside in the sense that $\mathbf{P}_x(X_{\zeta-}\in \mathbb{R})=0$ for q.e. $x\in \mathbb{R}$, where $\zeta$ is the lifetime of $(X_t)_{t\geq 0}$. Thanks to the killing transform and resurrected transform (see \cite[\S4]{LY17}), we can always assume $k=0$ when discussing the issues about regular Dirichlet subspaces (Cf. \cite[\S2.2.3]{LY15}). From now on and before \S\ref{SEC6}, we impose this assumption and in \S\ref{SEC6}, we shall briefly present the results with killing inside.
The interval $I_n$ is an invariant set of $(X_t)_{t\geq 0}$ in the sense that
\[
	\mathbf{P}_x(X_t\in I_n, \forall t\geq 0)=1,\quad x\in I_n.
\]
The restriction $X^{I_n}$ of $X$ to $I_n$ is an $m|_{I_n}$-symmetric diffusion enjoying irreducibility:
\[
	\mathbf{P}_x(\sigma_y<\infty)>0,\quad x,y\in I_n,
\]
where $\sigma_y$ is the first hitting time of $\{y\}$ relative to $(X_t)_{t\geq 0}$.
The scale function of $X^{I_n}$ is actually $\bs_n$. Note that the scale function $\bs_n$ is adapted to $I_n$ in the sense of $(\textbf{A}_R)$ and $(\textbf{B}_R)$, and this adapted condition is necessary for the regularity of $(\EE,\FF)$. Intuitively, this condition indicates that any finite endpoint of $I_n$ cannot be absorbing. Particularly, when $I_n$ is finite, $X^{I_n}$ must be recurrent, i.e. $$\mathbf{P}_x(\sigma_y<\infty)=1$$ for any $x,y\in I_n$.  Furthermore, every point outside these intervals is a trap of $(X_t)_{t\geq 0}$, that is
\[
	\mathbf{P}_x(X_t=x,\forall t\geq 0)=1,\quad x\in \left(\bigcup_{n\geq 1}I_n\right)^c.
\]
Therefore, the Dirichlet form $(\EE,\FF)$ in Theorem~\ref{THM21} is characterized by a set $\{(I_n, \bs_n): n\geq 1\}$ enjoying the following properties:
 \begin{itemize}
\item[(E1)] $\{I_n: n\geq 1\}$  are mutually disjoint.
\item[(E2)] For each $n$, $\bs_n$ is adapted to $I_n$, i.e. $\bs_n\in \SS_\infty(I_n)$.
\end{itemize}

Let us now give a definition.

\begin{definition}
A sequence of intervals $\{(I_n,\bs_n): n\geq 1\}$, with a scale function $\bs_n$ on $I_n$ for each $n$, is called \emph{(a class of) pre-effective intervals} if (E1) is satisfied, and \emph{(a class of) effective intervals}, if both (E1) and (E2) are satisfied.
 %It is called \emph{a class of pre-effective interval}, if only (E1) is satisfied.
\end{definition}

Therefore we could say that a regular and strongly local Dirichlet form $(\EE,\FF)$ is represented by a class of effective intervals.
We also call the  interval $I_n$ with an (adapted) scale function $\bs_n$ on $I_n$ or the pair $(I_n, \bs_n)$ a \emph{(pre-)effective interval} of $(\EE,\FF)$, if no confusions caused.

\section{Characterization of D-subspaces}\label{SEC3}

Let $(\EE,\FF)$ and $(\mathfrak{E},\mathfrak{F})$ be regular and strongly local Dirichlet forms on $L^2(\mathbb{R},m)$ with effective intervals $\{(I_n, \bs_n):n\geq 1\}$ and  %be another regular and strongly local Dirichlet form on $L^2(\mathbb{R},m)$ with the effective intervals
$\{(\mathtt{I}_k, \mathfrak{s}_k):k\geq 1\}$, respectively. The main purpose of this section is to present a necessary and sufficient condition on effective intervals  % $\{(I_n, \bs_n)\}$ and $\{(\mathtt{I}_k, \mathfrak{s}_k)\}$, which are equivalent to that
for $(\mathfrak{E},\mathfrak{F})$ to be a D-subspace of $(\EE,\FF)$ (in other words, $(\EE,\FF)$ is a D-extension of $(\mathfrak{E},\mathfrak{F})$).

Referring to \cite{FFY05} and \cite{FHY10}, when $(\EE,\FF)$ and $(\mathfrak{E}, \mathfrak{F})$ have only one effective interval $\mathbb{R}$, i.e. $I_1=\mathtt{I}_1=\mathbb{R}$, $(\mathfrak{E}, \mathfrak{F})$ is a D-subspace of $(\EE,\FF)$ if and only if $\mathfrak{s}_1$ is absolutely continuous with respect to $\bs_1$ and
\[
\frac{d\mathfrak{s}_1}{d\bs_1}=0\text{ or }1,\quad d\bs_1\text{-a.e.}
\]	
Clearly, when $\{x: d\mathfrak{s}_1/d\bs_1=0\}$ is of positive $d\bs_1$-measure, this D-subspace is proper. On the other hand, when $C_c^\infty(\mathbb{R})$ is a special standard core of $(\mathfrak{E},\mathfrak{F})$, this issue was explored in \cite[\S3]{LY17}, in which the idea of `interval-merge' operation was introduced.

In general, define
\[
\lambda_{\bs}:=\sum_{n\geq 1}d\bs_n,
\]
which is called  the \emph{scale measure} associated to effective intervals $\{(I_n, \bs_n): n\geq 1\}$. We shall write
\[
\EE(f,g)=\frac{1}{2}\int_\mathbb{R}\frac{df}{d\lambda_\ss}\frac{dg}{d\lambda_\ss}d\lambda_\ss
\]
for any $f,g\in \FF$, if no confusion caused.
 Note that each $d\bs_n$ is a Radon measure on $I_n$ and thus $\lambda_\bs$ is a $\sigma$-finite measure on $\mathbb{R}$ supported on the closure of $\bigcup_{n\geq 1}I_n$. Similarly, the scale measure associated to $\{(\mathtt{I}_k, \mathfrak{s}_k): k\geq 1\}$ is denoted by $\lambda_{\mathfrak{s}}$.
The following theorem could be treated as an extension of all results mentioned above. Note that any D-subspace of $(\EE,\FF)$ is also strongly local and characterized by another class of effective intervals. Thus this theorem is also a complete characterization of D-subspaces or D-extensions for a 1-dim symmetric diffusion.

\begin{theorem}\label{thm2.1}
Let $(\mathscr{E},\mathscr{F})$ and $(\mathfrak{E},\mathfrak{F})$ be two regular and strongly local Dirichlet forms on $L^2(\mathbb{R};m)$, with effective intervals  $\{(I_n, \bs_n): n\geq 1\}$ and $\{(\mathtt{I}_k, \mathfrak{s}_k): k\geq 1\}$ respectively. Further let $\lambda_{\bs}$ and $\lambda_\mathfrak{s}$ be the scale measures associated to $\{(I_n, \bs_n): n\geq 1\}$ and $\{(\mathtt{I}_k, \mathfrak{s}_k): k\geq 1\}$ respectively. Then $(\mathfrak{E},\mathfrak{F})$ is a D-subspace of $(\mathscr{E},\mathscr{F})$ on $L^2(\mathbb{R},m)$ if and only if the following conditions hold:
\begin{itemize}
\item[(1)] $\{\mathtt{I}_k: k\geq 1\}$ is coarser than $\{I_n: n\geq 1\}$ in the sense that for any $n$, $I_n\subset \mathtt{I}_k$ for some $k$.
\item[(2)] $\lambda_\mathfrak{s}\ll \lambda_\bs$ and
\begin{equation}\label{EQ3DLS}
	\frac{d\lambda_\mathfrak{s}}{d\lambda_\bs}=0\text{ or }1,\quad \lambda_\bs\text{-a.e.}
\end{equation}
\end{itemize}
\end{theorem}

\begin{proof}
For the sufficiency, we need only to prove
\[
\mathfrak{F}\subset \FF,\quad \EE(u,u)=\mathfrak{E}(u,u),\quad u\in \mathfrak{F}.
\]
Take a function $u\in \mathfrak{F}$. For any $n$, consider the restriction of $u$ to $I_n$, and denote it also by $u$ if no confusion caused. Let $\mathtt{I}_k$ be the interval in the first condition with $I_n\subset \mathtt{I}_k$. The second condition implies that $d\mathfrak{s}_k\ll d\bs_n$ and $d\mathfrak{s}_k/d\bs_n=0$ or $1$, $d\bs_n$-a.e. on $I_n$. Since $u\ll \mathfrak{s}_k$ on $\mathtt{I}_k$, it follows that  $u \ll \bs_n$ on $I_n$ and
\begin{equation}\label{EQ3IND}
 \int_{I_n}\left( \frac{du}{d\bs_n}\right)^2 d\bs_n= \int_{I_n}\left( \frac{du}{d\mathfrak{s}_k}\right)^2 \left( \frac{d\mathfrak{s}_k}{d\bs_n}\right)^2 d\bs_n =\int_{I_n}\left( \frac{du}{d\mathfrak{s}_k}\right)^2 d\mathfrak{s}_k.
\end{equation}
Note that $d\mathfrak{s}_k(\mathtt{I}_k\setminus \bigcup_{n\geq 1} I_n)=0$ since $\lambda_{\mathfrak{s}}\ll \lambda_\bs$. Then we have
\[
\sum_{I_n\subset \mathtt{I}_k}\int_{I_n}\left( \frac{du}{d\bs_n}\right)^2 d\bs_n=\int_{\mathtt{I}_k}\left( \frac{du|_{\mathtt{I}_k}}{d\mathfrak{s}_k}\right)^2 d\mathfrak{s}_k.
\]
Hence
\[
\sum_{n\geq 1}\int_{I_n}\left( \frac{du|_{I_n}}{d\bs_n}\right)^2 d\bs_n=\sum_{k\geq 1}\int_{\mathtt{I}_k}\left( \frac{du|_{\mathtt{I}_k}}{d\mathfrak{s}_k}\right)^2 d\mathfrak{s}_k.
\]
Therefore, we conclude $u\in \FF$ and $\EE(u,u)=\mathfrak{E}(u,u)$.

Next, we prove the necessity.  We first assert that any endpoint of each interval $\mathtt{I}_k$ cannot lie in the interior of any $I_n$. Indeed, let $\mathtt{I}_k=\langle \mathfrak{a}_k, \mathfrak{b}_k\rangle$ and suppose $\mathfrak{a}_k\in (a_n,b_n)$, where $I_n=\langle a_n, b_n \rangle$.
Consider the part Dirichlet forms $(\EE_{\mathring{I}_n}, \FF_{\mathring{I}_n})$ and $(\mathfrak{E}_{\mathring{I}_n}, \mathfrak{F}_{\mathring{I}_n})$ of $(\EE, \FF)$ and $(\mathfrak{E}, \mathfrak{F})$ on $\mathring{I}_n=(a_n, b_n)$. Clearly, $(\mathfrak{E}_{\mathring{I}_n}, \mathfrak{F}_{\mathring{I}_n})$ is a D-subspace of $(\EE_{\mathring{I}_n}, \FF_{\mathring{I}_n})$. However, $(\EE_{\mathring{I}_n}, \FF_{\mathring{I}_n})$ is irreducible while $(\mathfrak{E}_{\mathring{I}_n}, \mathfrak{F}_{\mathring{I}_n})$ is not irreducible. This contradicts \cite[Proposition~2.3~(3)]{LY16}. This assertion leads to $\mathring{I}_n\subset \mathtt{I}_k$ or $\mathring{I}_n\cap \mathtt{I}_k=\emptyset$ for any $k\geq 1$.  Moreover, $\mathring{I}_n\cap \left(\bigcup_{k\geq 1}\mathtt{I}_k\right)=\emptyset$ also contradicts the fact that $(\mathfrak{E}_{\mathring{I}_n}, \mathfrak{F}_{\mathring{I}_n})$ is a D-subspace of $(\EE_{\mathring{I}_n}, \FF_{\mathring{I}_n})$. Therefore, we conclude that $\mathring{I}_n\subset \mathtt{I}_k$ for some $k$. Note that $(\EE_{\mathring{I}_n}, \FF_{\mathring{I}_n})$ and $(\mathfrak{E}_{\mathring{I}_n}, \mathfrak{F}_{\mathring{I}_n})$ are both irreducible on $L^2({\mathring{I}_n},m|_{\mathring{I}_n})$ and their scale functions are $\bs_n$ and $\mathfrak{s}_k$ (restricted to ${\mathring{I}_n}$) respectively. Thus it follows from \cite{FFY05} (or \cite[Lemma~3.1]{LY17}) that
\begin{equation}\label{EQ3SKS}
	d\mathfrak{s}_k\ll d\bs_n, \quad \frac{d\mathfrak{s}_k}{d\bs_n}=0\text{ or }1,\quad d\bs_n\text{-a.e. on }{\mathring{I}_n}.
\end{equation}
For the first condition, it suffices to prove $I_n\subset \mathtt{I}_k$. Suppose that $I_n\nsubseteq \mathtt{I}_k$, which implies that $a_n=\mathfrak{a}_k, a_n\in I_n\setminus \mathtt{I}_k$ or $b_n=\mathfrak{b}_k, b_n\in I_n\setminus \mathtt{I}_k$. It follows that either $\bs_n(a_n)>-\infty, \mathfrak{s}_k(\mathfrak{a}_k)=-\infty$ or $\bs_n(b_n)<\infty, \mathfrak{s}_k(\mathfrak{b}_k)=\infty$. Either of them contradicts \eqref{EQ3SKS}.

Finally, for the second condition, we need only to show that
\begin{equation}\label{EQ3DSK}
	d\mathfrak{s}_k\left(\mathtt{I}_k\setminus \left(\bigcup_{n\geq 1}I_n\right)\right)=0, \quad k\geq 1.
\end{equation}
In fact, for any $u\in \mathfrak{F}$ with $u|_{\mathtt{I}^c_k}=0$, we have
\[
	\EE(u,u)=\frac{1}{2}\sum_{I_n\subset \mathtt{I}_k}\int_{I_n} \left(\frac{du}{d\bs_n}\right)^2d\bs_n,\quad \mathfrak{E}(u,u)=\frac{1}{2}\int_{\mathtt{I}_k}\left(\frac{du}{d\mathfrak{s}_k}\right)^2d\mathfrak{s}_k.
\]
Similar to \eqref{EQ3IND}, we can deduce from $\EE(u,u)=\mathfrak{E}(u,u)$ that
\[
\int_{\mathtt{I}_k\setminus \left(\bigcup_{n\geq 1}I_n\right)}\left(\frac{du}{d\mathfrak{s}_k}\right)^2d\mathfrak{s}_k=0.
\]
Therefore, we conclude that \eqref{EQ3DSK} holds. That completes the proof.
\end{proof}

\begin{remark}\label{RM32}
We shall present several remarks which contain some interesting facts, regarding two conditions in the theorem above.
Suppose that the conclusion of Theroem~\ref{thm2.1} holds.
\begin{itemize}
\item[(1)] The first condition is equivalent to say that $\{I_n: n\geq 1\}$ is finer than $\{\mathtt{I}_k: k\geq 1\}$. Since both $\{I_n: n\geq 1\}$ and $\{\mathtt{I}_k: k\geq 1\}$ satisfy (E1), the finer set $\{I_n: n\geq 1\}$ could be divided into % several classes
\[
	\mathfrak{I}_k:=\{I_n: I_n\subset \mathtt{I}_k, n\geq 1\},\quad k\geq 1.
\]
Clearly, $\{\mathfrak{I}_k: k\geq 1\}$ are mutually disjoint. Furthermore, we claim that for each $k$, the intervals in $\mathfrak{I}_k$ are topologically dense in $\mathtt{I}_k$ in the sense that
\begin{equation}\label{EQ3IKJ}
	\mathtt{I}_k\setminus \left(\bigcup_{J\in \mathfrak{I}_k} J \right)
\end{equation}
is nowhere dense. In fact, suppose that the open set $(a,b)$ is contained in the closure of \eqref{EQ3IKJ}. Then $$(a,b)\subset \bar{\mathtt{I}}_k\setminus \left(\bigcup_{J\in \mathfrak{I}_k}\mathring{J}\right),$$ where $\bar{\mathtt{I}}_k$ is the closure of $\mathtt{I}_k$ and  $\mathring{J}$ is the interior of $J$. It follows that $\lambda_{\mathfrak{s}}\left((a,b) \right)>0$ but $\lambda_{\bs}\left((a,b) \right)=0$, which contradicts the absolute continuity $\lambda_{\mathfrak{s}}\ll \lambda_{\bs}$.
\item[(2)] The second condition is equivalent to the statement that the $d\mathfrak{s}_k$-measure of \eqref{EQ3IKJ} is equal to $0$ for any $k$ and whenever $I_n\in \mathfrak{I}_k$, it holds that $\mathfrak{s}_k\ll \bs_n$ on $I_n$ and
\[
	\frac{d\mathfrak{s}_k}{d\bs_n}=0\text{ or }1,\quad d\bs_n\text{-a.e. on }I_n.
\]
\item[(3)] $(\EE,\FF)=(\mathfrak{E},\mathfrak{F})$ if and only if $\{I_n:n\geq 1\}=\{\mathtt{I}_k: k\geq 1\}$ and $\lambda_{\bs}=\lambda_{\mathfrak{s}}$. The condition they share the same scale measure alone,
     %If we assume only that their scale measures coincide (i.e.
    $\lambda_{\bs}=\lambda_{\mathfrak{s}}$), can not ensure $(\mathfrak{E},\mathfrak{F})=(\EE,\FF)$. An example is given in the end of this section.

\item[(4)] When $(\EE,\FF)$ has only one effective interval $\mathbb{R}$, $(\mathfrak{E},\mathfrak{F})$ also has only one effective interval $\mathbb{R}$ and the first condition naturally holds. Then the scale measures reduce to the scale functions on $\mathbb{R}$ and the second condition is nothing but the one introduced in \cite{FFY05}.
\end{itemize}
\end{remark}

We end this section with two examples. The first example is given in \cite[Example~3.12 and Remark~3.15]{LY17}, in which $(\EE,\FF)$ and $(\mathfrak{E},\mathfrak{F})$ both have only one effective interval, and $I_1=(0,1]$ and $\mathtt{I}_1=[0,1]$. Clearly, $I_1\subset \mathtt{I}_1$. Note that $\bs_1$ is adapted to $I_1$ and $\mathfrak{s}_1$ is adapted to $\mathtt{I}_1$, which implies $\bs_1(0)=-\infty$ and $\mathfrak{s}_1(0)>-\infty$. Due to \eqref{EQ3DLS}, such $\bs_1$ and $\mathfrak{s}_1$ are constructed in \cite[Example~3.12]{LY17}.
Another example below is given in \cite[Example~3.20]{LY16} and \cite[Example~3.8]{LY17}.

\begin{example}\label{EXA33}
Let $m$ be the Lebesgue measure on $\mathbb{R}$ and $(\mathfrak{E},\mathfrak{F})$ be the Dirichlet form of 1-dim Brownian motion. The effective intervals of $(\EE,\FF)$ are given as follows. Let $K$ be the standard Cantor set in $[0,1]$. Set $U:=K^c$ and write $U$ as a union of disjoint open intervals:
\[U=\bigcup_{n \geq 1}(a_n,b_n)\]
where $(a_1,b_1)=(-\infty,0), (a_2,b_2)=(1,\infty)$. Let $I_1:=(-\infty,0]$, $I_2:=[1,\infty)$ and $I_n:=[a_n,b_n]$ for any $n\geq 3$. For each $n$, define the scale function $\bs_n(x)=x-e_n$ on $I_n$ ($e_n\in I_n$ as in \eqref{midpoint}). Since $\mathtt{I}_1=\mathbb{R}$ and $\mathfrak{s}_1(x)=x$, it follows that $I_n\subset \mathtt{I}_1$, and $d\bs_n$ and $d\mathfrak{s}_1$ coincide on $I_n$. Therefore, $\lambda_\bs=\lambda_\mathfrak{s}$, which is the Lebesgue measure on $\mathbb{R}$, as an example promised in Remark~\ref{RM32}(3).
\end{example}

\section{Road map to attain D-subspaces}\label{SEC4}

Fix a regular and strongly local Dirichlet form $(\EE,\FF)$ on $L^2(\mathbb{R},m)$ with effective intervals $\{(I_n, \bs_n): n\geq 1\}$ throughout this section. Though Theorem~\ref{thm2.1} in \S\ref{SEC3} completely characterizes the D-subspaces of $(\EE,\FF)$, it does not tell how to construct a D-subspace from $(\EE,\FF)$.  % in some concrete procedures.
In this section, we shall draw an intuitive road map  %with two essential operations on effective intervals
to attain all possible D-subspaces of $(\EE,\FF)$.
Note that a D-subspace of $(\EE,\FF)$ is also expressed by another class of effective intervals. Thus essentially, it suffices to construct a new class of effective intervals.
In the following, we shall introduce two kinds of operations on effective intervals. Keep in mind that after each operation, the obtained  effective intervals would always enjoy the properties (E1) and (E2) in \S\ref{SEC2}. Finally, the principal theorem shows that every D-subspace can be attained by these two operations.

\subsection{Scale-shrink operation}

The first kind of operation is called \emph{scale-shrink}. It is composed of a key step and an additional step.

\subsubsection{Key step}

The key step based on \eqref{EQ3DLS} is to construct a new scale measure. Define a scale function $\bar{\bs}_n\in \mathbf{S}(I_n)$ such that $d\bar{\bs}_n\ll d\bs_n$ and
\begin{equation}\label{EQ4DSN}
\frac{d\bar{\bs}_n}{d\bs_n}=0 \text{ or }1,\quad d\bs_n\text{-a.e. on }I_n
\end{equation}
for each $n$ and set
\[
	\lambda_{\bar{\bs}}:=\sum_{n\geq 1} d\bar{\bs}_n,
\]
which is the scale measure associated to  $\{(I_n,\bar{\bs}_n): n\geq 1\}$. The measure $\lambda_{\bar{\bs}}$
is called a {shrinking} of $\lambda_\bs$.

This step produces a valid scale function on each $I_n$, but the sequence of pairs $\{(I_n,\bar{\bs}_n): n\geq 1\}$ might break (E2), adaptedness, so that it is possibly only a class of pre-effective intervals. In practice, when $I_n$ is closed, \eqref{EQ4DSN} implies that $\bar{\bs}_n$ is still adapted to $I_n$. However, when $I_n$ is not closed (semi-closed or open),  Lemma~\ref{resc} below tells us one can always find scale function $\bar{\bs}_n$ on $I_n$, which satisfies \eqref{EQ4DSN} and is not adapted to $I_n$.
A concrete example is given in \cite[Example~3.12]{LY17} and also mentioned in \S\ref{SEC3}.

\begin{lemma}\label{resc}
Let $J$ be an interval. Take $\bs\in \mathbf{S}(J)$ and fix a constant $\varepsilon>0$. Then there exists a scale function $\mathfrak{s}$ on $J$ such that $d\mathfrak{s}(J)<\varepsilon$, $d\mathfrak{s}\ll d\bs$ and
\begin{equation}\label{a.c.1}
\frac{d\mathfrak{s}}{d\bs}=0 \text{ or }  1, \quad d\bs\text{-a.e.}
\end{equation}
\end{lemma}

\begin{proof}
Set
\[
G:=\left\{\mathop{\bigcup}\limits_{k\geq 1} (q_k-r_k,q_k+r_k)\right\}\cap J,
\]
where $\{q_k:k\geq 1\}$ is the set of rational numbers in $J$, and $\{r_k:k\geq 1\}$ is a sequence of real number such that for each $k$, $\bs(q_k+r_k)-\bs(q_k-r_k)< \varepsilon2^{-k}$. Take a fixed point $e\in J$ as in \eqref{midpoint} and define
\[
	\mathfrak{s}(x)=\int_e^x 1_{G}(y)d\bs(y),\quad x\in J.
\]
We assert that $\mathfrak{s}$ satisfies all conditions. Clearly, $\mathfrak{s}$ is absolutely continuous with respect to $\bs$ and \eqref{a.c.1} holds.
Since $G$ is a dense open subset of $J$ and $\bs$ is strictly increasing, it follows that $\mathfrak{s}$ is strictly increasing and thus $\mathfrak{s}\in \SS(J)$.
Finally,
\[d\mathfrak{s}(J)=\int_J 1_{G}(y)d\bs(y) \leq \sum_{k\geq 1}\int_{(q_k-r_k,q_k+r_k)}d\bs <\varepsilon.
\]
That completes the proof.
\end{proof}

 %We introduce the following definition for convenience. Therefore after the first step of scale-shrink operation, the sequence of pairs $\{(I_n,\bar{\bs}_n):n\ge 1\}$ is only a class of pre-effective intervals.

 %\begin{definition} A sequence of disjoint intervals with a scale function on each interval, i.e., satisfying (E1) only, is called a class of \emph{pre-effective intervals}.
 %\end{definition}

Before proceeding, we should make pre-effective intervals $\{(I_n,\bar{\bs}_n): n\geq 1\}$ be effective.
A naive way to restore (E2) after the key step is as follows. Whenever $\bar{\bs}_n$ is not adapted to $I_n$, the endpoints, in trouble, of $I_n$ are added to $I_n$ and a new interval $I^{\bar{\bs}}_n$ obtained so that $\bar{\bs}_n$ is adapted to $I^{\bar{\bs}}_n$. Precisely speaking, for any $n\geq 1$, set
\begin{equation}\label{EQ4ISN}
	I^{\bar{\bs}}_n:=\langle a_n, b_n\rangle,
\end{equation}
where $a_n\in I^{\bar{\bs}}_n$ (resp. $b_n\in I^{\bar{\bs}}_n$) if and only if $a_n+\bar{\bs}_n(a_n)>-\infty$ (resp. $b_n+\bar{\bs}_n(b_n)<\infty$). Note that $I_n\subset I^{\bar{\bs}}_n$.
However doing so  %does not work properly, because this step
might break (E1). In other words, some intervals could intersect with each other. Surely, we could merge the intersected intervals into a new one and maintain the scale measure. But this `interval-merge' seems not so clear and new (and endless) troubles might appear. The following examples show us a rough observation about this dilemma.

\begin{example}\label{EXA49}
In the example below, the scale function $\bs(x)=x-e$ on an arbitrary interval ($e$ is a fixed point given by \eqref{midpoint}) is called the natural scale function.  	
\begin{itemize}
\item[(1)] Consider a simple example with only two effective intervals: $I_1=[-1,0)$ and $I_2=(0,1]$.  Without loss of generality, assume that the scale functions $\bar{\bs}_1, \bar{\bs}_2$ in the key step are both the natural scale functions. Clearly, (E2) does not hold for the new scales. As in \eqref{EQ4ISN}, set $I^{\bar{\bs}}_1:=[-1,0]$ and $I^{\bar{\bs}}_2:=[0,1]$. Then $\bar{\bs}_1, \bar{\bs}_2$ are adapted to
$I^{\bar{\bs}}_1$ and $I^{\bar{\bs}}_2$ respectively but $I^{\bar{\bs}}_1\cap I^{\bar{\bs}}_2\neq \emptyset$. In this case, the merging step is easy and clear. Merge $I^{\bar{\bs}}_1$ and $I^{\bar{\bs}}_2$ into a new interval $[-1,1]$ and the induced scale function on $[-1,1]$ remains to be the natural scale function.
\item[(2)] Consider the following intervals: $I_1=(0,1)$, $I_2=(-1,-1/2)$,  and $I_n=(-1/(n-1), -1/n)$ for any $n\geq 3$. Let $\bar{\bs}_n$ be the natural scale function on $I_n$. As in \eqref{EQ4ISN}, we have
\[
	I^{\bar{\bs}}_2=[-1,-1/2],\quad I^{\bar{\bs}}_n=[-1/(n-1), -1/n],\quad n\geq 3.
\]
These intervals constitute a chain one by one. Note that $I^{\bar{\bs}}_1=[0,1]$ and $I^{\bar{\bs}}_1\cap I^{\bar{\bs}}_n=\emptyset$ for any $n\geq 2$. Thus it seems reasonable to divide them into two classes:
\[
	\mathfrak{I}_1=\{I^{\bar{\bs}}_1\},\quad \mathfrak{I}_2=\{I^{\bar{\bs}}_n: n\geq 2\},
\]
and then merge $\mathfrak{I}_2$ into a new interval
\begin{equation}\label{EQ4INI}
	\mathtt{I}_2:= \bigcup_{n\geq 2} I^{\bar{\bs}}_n=[-1, 0).
\end{equation}
However, the natural scale function is not adapted to $\mathtt{I}_2$. Thus we need to mimic \eqref{EQ4ISN} again and set $\mathtt{I}^{\bar{\bs}}_2:=[-1,0]$. Then the new trouble appears: $I^{\bar{\bs}}_1$ intersects with $\mathtt{I}^{\bar{\bs}}_2$. Finally, we could merge $I^{\bar{\bs}}_1$ and $\mathtt{I}^{\bar{\bs}}_2$ into $[-1,1]$ and the natural scale function is eventually adapted to $[-1,1]$.
\item[(3)] The repair procedures in the second example might be infinite. Mimicking the second example, we can always find a sequence of open intervals one by one, which converge in both directions. For example, let
\[
\begin{aligned}
	& I_1=(-1/2, 0),\quad  I_2=(0, 1/2), \\
	&I_{2n+1}=\left(\frac{1}{n+2}-1, \frac{1}{n+1} -1\right),\\
  &I_{2n+2}= \left(1-\frac{1}{n+1}, 1-\frac{1}{n+2}\right),\quad n\geq 1.
\end{aligned}
\]
Then $I_{2n+1}\downarrow -1$ and $I_{2n+2}\uparrow 1$ as $n\rightarrow \infty$. We call them a sequence of intervals in the first type. Now take a decreasing sequence of points $a_n\downarrow a$ and an increasing sequence of points $b_n\uparrow b$ with $a_1=b_1$. For each $n$, take a sequence of open intervals in the first type between $(b_n, b_{n+1})$ (resp. $(a_{n+1}, a_n)$), and $b_n, b_{n+1}$ (resp. $a_{n+1}, a_n$) are two convergent points. All the intervals between $a$ and $b$ are called a sequence of intervals in the second type. Similarly, we can build  sequences in the $m$-th type for any integer $m$. When dealing with these intervals by mimicking the procedures of the second example, we need to repeat the extension (as \eqref{EQ4ISN}) and merging (as \eqref{EQ4INI}) of intervals infinite many times.
\end{itemize}
\end{example}

 %We shall prove that it is possible to make pre-effective intervals be effective intervals after another operation.
 %Roughly speaking, if we intend to replace the scale function $\bs_n$ on $I_n$ with a shrunken scale function $\bar{\bs}_n$ for any $n$, the effective intervals need to be renewed to keep adapted.

\subsubsection{Minimal interval-merge}

 To overcome this problem, we need to design an additional step to complete scale-shrink operation. Heuristically speaking, these `virtually connected' intervals (in Example~\ref{EXA49}) should be glued together to make the scale function be adapted. This step is called `minimal interval-merge', because it is a minimal operation to get the job done.  %Intuitively we have to. % is a real method to repair the relation between $I_n$ and $\bar{\bs}_n$ for any $n\geq 1$.

To show the details, let us introduce a new conception.  Recall that $e_n$ is a fixed point in $(a_n,b_n)$ for $n\geq 1$ and $[e_i,e_j]$ denotes the interval ended by $e_j$ and $e_i$ no matter which is bigger.
The following definition is stated for $\{(I_n,\bar{\bs}_n):n\ge 1\}$, but may actually be applied to any class of pre-effective intervals. This means that the minimal interval-merge may be applied to any class of pre-effective intervals.

\begin{definition}\label{DEF43}
We say that two intervals $I_i$ and $I_j$  are \emph{tightly scale-connected}, or $I_i$ is \emph{tightly scale-connected} to $I_j$, with respect to the scale measure $\lambda_{\bar{\bs}}$, if
\begin{itemize}
\item[(1)] $[e_i, e_j]\setminus \left(\bigcup_{n\geq 1}I_n\right)$ is at most countable;
\item[(2)] $\lambda_{\bar{\bs}}([e_i,e_j])<\infty$.
\end{itemize}
\end{definition}

\begin{remark}\label{RM44}
Clearly, all the intervals in each example of Example~\ref{EXA49} are tightly scale-connected. Let us give two further remarks for this definition.
\begin{itemize}
\item[(1)]
Condition (1) in Definition~\ref{DEF43} implies that $[e_i, e_j]\setminus \left(\bigcup_{n\geq 1}I_n\right)$ is nowhere dense. In fact, it is evident that $$\overline{[e_i, e_j]\setminus \left(\bigcup_{n\geq 1}I_n\right)}\subset [e_i, e_j]\setminus \left(\bigcup_{n\geq 1}\mathring{I}_n\right).$$ Condition (1) above implies that $[e_i, e_j]\setminus \left(\bigcup_{n\geq 1}\mathring{I}_n\right)$ has at most countable points, so that it contain no open non-empty intervals. Thus $[e_i, e_j]\setminus \left(\bigcup_{n\geq 1}I_n\right)$ is nowhere dense.
\item[(2)] The phenomena of tight scale-connection can only exist in the case of pre-effective intervals. For effective intervals $\{(I_n,\bs_n):n\ge 1\}$, we claim that any two different intervals $I_i$ and $I_j$ are not tightly scale-connected with respect to its scale measure $\lambda_\bs$. In fact, suppose two different intervals $I_i$ and $I_j$ are tightly scale-connected. Then $\lambda_{\bs}([e_i, e_j])<\infty$ implies that the endpoints of $\{I_n: n\geq 1\}$ between $e_i$ and $e_j$ must be closed. Particularly, the intervals between $I_i$ and $I_j$ are closed. Since $I_i\cap I_j=\emptyset$, it follows that the intervals between $e_i$ and $e_j$ have a Cantor-type structure. This indicates $[e_i, e_j]\setminus \left(\bigcup_{n\geq 1}I_n\right)$ has uncountable points, which contradicts condition (1) in Definition~\ref{DEF43}.
\end{itemize}
\end{remark}

The tight scale-connection is an equivalence relation for intervals $\{I_n: n\geq 1\}$. Denote all the equivalence classes induced by tight scale-connection by
\[
\{\mathfrak{I}_k: k\geq 1\},
\]
where $\mathfrak{I}_k$ is a subset of $\{I_n: n\geq 1\}$ for each $k$, in which any two intervals are tightly scale-connected. Note that if $I_i$ is tightly scale-connected to $I_j$, then any interval located between $I_i$ and $I_j$ must be tightly scale-connected to them. Thus each equivalence class looks like a `connected' cluster of intervals and there are at most countable points between these intervals.

 %Then $\{\mathfrak{I}^2_k: k\geq 1\}$ is a refinement of $\{\mathfrak{I}_k^1: k\geq 1\}$.

The additional step, called \emph{minimal interval-merge}, is defined as follows. For each $k\geq 1$, set
\begin{equation}\label{EQ4AKX}
\mathfrak{a}_k:=\inf\{x\in J: J\in \mathfrak{I}_k\},\quad \mathfrak{b}_k:=\sup\{x\in J: J\in \mathfrak{I}_k\}.
\end{equation}
Let $\mathfrak{s}_k$ be the increasing function induced by the measure $\lambda_{\bar{\bs}}|_{(\mathfrak{a}_k,\mathfrak{b}_k)}$ and  %$d\mathfrak{s}_k=\lambda_{\bar{\bs}}|_{(\mathfrak{a}_k,\mathfrak{b}_k)}$,
such that $\mathfrak{s}_k(\mathfrak{e}_k)=0$, where $\mathfrak{e}_k$ is a fixed point in $(\mathfrak{a}_k,\mathfrak{b}_k)$ (Cf. \eqref{midpoint}). More precisely,
$$\mathfrak{s}_k(x):=\int_{\mathfrak{e}_k}^xd\lambda_{\bar{\bs}}|_{(\mathfrak{a}_k,\mathfrak{b}_k)},\quad
x\in(\mathfrak{a}_k,\mathfrak{b}_k).$$
Clearly, $\mathfrak{s}_k\in \mathbf{S}\left((\mathfrak{a}_k,\mathfrak{b}_k)\right)$ by Definition~\ref{DEF43} and Remark~\ref{RM44}. Define now
\begin{equation}\label{EQ4IKA}
	\mathtt{I}_k:=\langle \mathfrak{a}_k, \mathfrak{b}_k \rangle,
\end{equation}
where $\mathfrak{a}_k\in \mathtt{I}_k$ (resp. $\mathfrak{b}_k\in \mathtt{I}_k$) if and only if $\mathfrak{a}_k+\mathfrak{s}_k(\mathfrak{a}_k)>-\infty$ (resp. $\mathfrak{b}_k+\mathfrak{s}_k(\mathfrak{b}_k)<\infty$).
For each $k$, the pair $(\mathtt{I}_k, \mathfrak{s}_k)$ is called \emph{{merging} of equivalence class} $\mathfrak{I}_k$ \emph{relative to} $\lambda_{\bar{\bs}}$.

\begin{remark}\label{RM45}
 %\begin{itemize}
 %\item[(1)] Since the effective intervals of subspace are coarser than those of $(\EE,\FF)$ by Theorem~\ref{thm2.1}, the merge as above is the only way  to attain a subspace.
When $\bar{\bs}_n=\bs_n$ for any $n\geq 1$, two scale measures are the same. It is said in Remark~\ref{RM44}~(2) that any two intervals $I_i$ and $I_j$ are not tightly scale-connected with respect to the original scale measure $\lambda_{\bs}$. Hence the minimal interval-merge takes no action,  %merge $\{(\mathtt{I}_k, \mathfrak{s}_k): k\geq 1\}$ are nothing but $\{(I_n, \bs_n): n\geq 1\}$,
i.e. $\{(\mathtt{I}_k, \mathfrak{s}_k): k\geq 1\}=\{(I_n, \bs_n): n\geq 1\}$.
 %\end{itemize}
\end{remark}

The following lemma asserts that $\{(\mathtt{I}_k,\mathfrak{s}_k): k\geq 1\}$, obtained by scale-shrink operation, is also a class of effective intervals, whose associated scale measure is $\lambda_{\bar{\bs}}$ defined in the key step. % and it is essentially what we want by a scale-shrink operation.

\begin{lemma}\label{LM46}
The sequence $\{(\mathtt{I}_k,\mathfrak{s}_k): k\geq 1\}$ is a class of effective intervals. % i.e. satisfies (E1) and (E2).
In addition, the scale measure associated to $\{(\mathtt{I}_k,\mathfrak{s}_k): k\geq 1\}$ is exactly $\lambda_{\bar{\bs}}$, i.e. $\lambda_{\mathfrak{s}}=\lambda_{\bar{\bs}}$.
\end{lemma}
\begin{proof}
It suffices to prove that $\{\mathtt{I}_k: k\geq 1\}$ are mutually disjoint.  Suppose that $\mathtt{I}_1\cap \mathtt{I}_2\neq \emptyset$ and $\mathfrak{b}_1=\mathfrak{a}_2$. Take an interval $I_i\in \mathfrak{I}_1$ and another interval $I_j\in \mathfrak{I}_2$. We assert that $I_i$ is tightly scale-connected to $I_j$, which contradicts the definition of equivalence classes. In fact, since $\mathfrak{b}_1\in \mathtt{I}_1$ and $\mathfrak{a}_2\in \mathtt{I}_2$, it follows that $\lambda_{\bar{\bs}}([e_i, e_j])<\infty$. On the other hand, we can take a sequence of increasing intervals $\{I_{p_m}: m\geq 1\}\subset \mathfrak{I}_1$ with $I_{p_1}=I_i$ such that $a_{p_m}, b_{p_m} \uparrow \mathfrak{b}_1$ as $m\rightarrow \infty$ and another sequence of decreasing intervals $\{I_{q_m}: m\geq 1\}\subset \mathfrak{I}_2$ with $I_{q_1}=I_j$ such that $a_{q_m}, b_{q_m} \downarrow \mathfrak{a}_2$ as $m\rightarrow \infty$. Since both $$(b_{p_m}, a_{p_{m+1}})\setminus\left( \bigcup_{n\geq 1} I_n\right)\ \text{and}\  (b_{q_{m+1}}, a_{q_{m}})\setminus\left( \bigcup_{n\geq 1} I_n\right)$$ contain at most countable points, we then conclude that $[e_i,e_j]\setminus \left( \bigcup_{n\geq 1} I_n\right)$ contains at most countable points. That completes the proof.
\end{proof}

\subsubsection{Scale-shrink operation}

We may summarize scale-shrink operation in the following definition.

\begin{definition}\label{DEF48}
A scale-shrink operation on $\{(I_n, \bs_n): n\geq 1\}$ proceeds as follows:
\begin{itemize}
\item[(1)] Key step. Identify a new scale measure $\lambda_{\bar{\bs}}$ by defining a new scale function $\bar{\bs}_n$  satisfying \eqref{EQ4DSN}  on each interval $I_n$.
\item[(2)] Minimal interval-merge. Divide $\{I_n: n\geq 1\}$ into several equivalence classes $\{\mathfrak{I}_k: k\geq 1\}$ according to tight scale-connection with respect to $\lambda_{\bar{\bs}}$, and for each $k\geq 1$, merge $\mathfrak{I}_k$ with scale functions into a new pair $(\mathtt{I}_k, \mathfrak{s}_k)$.
\end{itemize}
This operation gives a new class of effective intervals $\{(\mathtt{I}_k, \mathfrak{s}_k): k\geq 1\}$.
\end{definition}

Let us give two remarks for this operation. We first present a proposition, which ensure that it is a right approach to attain D-subspaces. Let  $(\mathfrak{E},\mathfrak{F})$  be the Dirichlet form given by \eqref{EQ2FULR} with effective intervals $\{(\mathtt{I}_k,\mathfrak{s}_k): k\geq 1\}$ in Definition~\ref{DEF48}.

\begin{proposition}\label{PRO48}
 $(\mathfrak{E},\mathfrak{F})$ is a D-subspace of $(\EE,\FF)$.
\end{proposition}
\begin{proof}
 It suffices to verify two conditions in Theorem~\ref{thm2.1}. For any $n\geq 1$, $I_n$ belongs to an equivalence class, say $\mathfrak{I}_k$. We assert $I_n\subset \mathtt{I}_k$. In fact, it follows from \eqref{EQ4AKX} and \eqref{EQ4IKA} that $\mathring{I}_n\subset \mathtt{I}_k$. Suppose $a_n\in I_n$ but $a_n\notin \mathtt{I}_k$. This implies that $\mathfrak{a}_k=a_n$ and $\bar{\bs}_n(a_n)>-\infty, \mathfrak{s}_k(\mathfrak{a}_k)=-\infty$, which contradicts the definition of $\mathfrak{s}_k$. Thus $\{\mathtt{I}_k: k\geq 1\}$ is coarser than $\{I_n: n\geq 1\}$. On the other hand, from the first term of Definition~\ref{DEF43}, we can deduce that $\mathtt{I}_k\setminus \left(\bigcup_{n\geq 1}I_n \right)$ is at most countable. Since $\lambda_{\bar{\bs}}$ charges no singleton, we have
\[
	d\mathfrak{s}_k\left(\mathtt{I}_k\setminus \left(\bigcup_{n\geq 1}I_n \right)\right)=\lambda_{\bar{\bs}}\left(\mathtt{I}_k\setminus \left(\bigcup_{n\geq 1}I_n \right)\right)=0.
\]
Then it follows from Remark~\ref{RM32}~(2) and \eqref{EQ4DSN} that \eqref{EQ3DLS} holds.
That completes the proof.
\end{proof}

Next, we note that the minimal interval-merge overcomes the dilemma illustrated in Example~\ref{EXA49}. Roughly speaking, it takes essentially the fewest actions (such as \eqref{EQ4ISN} and \eqref{EQ4INI}) to attain a new class of effective intervals under the given scale measure $\lambda_{\bar{\bs}}$. The following proposition makes the rough idea above rigorous.  % a strict assertion about the conclusion above.
Note that under the fixed scale measure, a D-subspace is always characterized by a coarser class of intervals, and thus a smaller D-subspace requests more merging.

\begin{proposition}\label{PRO410}
Let $(\EE,\FF)$ and $(\mathfrak{E}, \mathfrak{F})$ be in Proposition~\ref{PRO48}. If $(\EE',\FF')$ is another D-subspace of $(\EE,\FF)$ with the same scale measure as $(\mathfrak{E}, \mathfrak{F})$, then $(\EE',\FF')$ is a D-subspace of $(\mathfrak{E},\mathfrak{F})$ on $L^2(\mathbb{R},m)$.
\end{proposition}
\begin{proof}
Let $\{I'_m: m\geq 1\}$ be the effective intervals of $(\EE',\FF')$. By Theorem~\ref{thm2.1}, it suffices to prove that $\{I'_m: m\geq 1\}$ is coarser that $\{\mathtt{I}_k: k\geq 1\}$. Since $(\EE',\FF')$ is a D-subspace of $(\EE,\FF)$, it follows from Remark~\ref{RM32}~(1) that we can divide $\{I_n: n\geq 1\}$ into classes:
\[
\mathcal{I}_m:=\{I_n: I_n\subset I'_m,n\geq 1\},\quad m\geq 1.
\]
We assert that for any $k$, $\mathfrak{I}_k\subset \mathcal{I}_m$ for some $m$, which implies that $\mathtt{I}_k\subset I'_m$. In fact, suppose that $I_1, I_2\in \mathfrak{I}_k$ but $I_1\in \mathcal{I}_1, I_2\in \mathcal{I}_2$. Since $\lambda_\mathfrak{s}([e_1,e_2])<\infty$, it follows that the endpoints of $\{I'_m: m\geq 1\}$ between $e_1$ and $e_2$ must be closed. Mimicking Remark~\ref{RM44}~(2), we obtain that $[e_1,e_2]\setminus \bigcup_{m\geq 1}I'_m$ has uncountable points. Since $\bigcup_{n\geq 1}I_n\subset \bigcup_{m\geq 1}I'_m$, we know that $[e_1,e_2]\setminus \bigcup_{n\geq 1}I_n$ also has uncountable points, which contradicts the tight scale-connection of $I_1$ and $I_2$. That completes the proof.
\end{proof}

\subsection{Optional interval-merge operation}\label{SEC42}

The scale-shrink operation discussed above produces the largest D-subspace with the given scale measure, as indicated in Proposition~\ref{PRO410}. In order to attain all D-subspaces with the given scale measure, we need more merging. We call this operation the \emph{optional interval-merge}.

\subsubsection{Maximal interval-merge}

Let us start with a special case of optional interval-merge, the so-called maximal interval-merge. Since the scale measure is fixed, there is no loss of generality to start from effective intervals $\{(I_n, \bs_n): n\geq 1\}$ and scale measure $\lambda_\bs$.  The following notion is the basis of maximal interval-merge.

\begin{definition}\label{DEF42}
We call two intervals $I_i$ and $I_j$  are \emph{loosely scale-connected}, or $I_i$ is \emph{loosely scale-connected} to $I_j$, with respect to the scale measure $\lambda_\bs$, if
\begin{itemize}
\item[(1)] $[e_i, e_j]\setminus \left(\bigcup_{n\geq 1}I_n\right)$ is nowhere dense;
\item[(2)] $\lambda_{\bs}([e_i,e_j])<\infty$.
\end{itemize}
\end{definition}

\begin{remark}
 %This conception is borrowed from \cite[Definition~3.5]{LY17}, in which the first condition is replaced by a stronger one: $[e_i, e_j]\setminus \left(\bigcup_{n\geq 1}I_n\right)$  is of zero Lebesgue measure.
Tight scale-connection implies loose scale-connection by Remark~\ref{RM44}, but not vice versa.   %when $\bs_n$ is adapted to $I_n$,   However, the loose scale-connection is a weaker condition.
For example, the effective intervals in Example~\ref{EXA33} are loosely scale-connected with each other. However, Remark~\ref{RM44}~(2) tells us any interval $I_n$ is tightly scale-connected to itself only.
\end{remark}

Like tight scale-connection, loose scale-connection is also an equivalence relation on intervals $\{I_n: n\geq 1\}$. Analogously, denote all the equivalence classes of $\{I_n: n\geq 1\}$ induced by loose scale-connection by
\begin{equation}\label{EQ4IKK}
\{\mathfrak{I}_k: k\geq 1\}.
\end{equation}
 %where $\mathfrak{I}_k$ is a subset of $\{I_n: n\geq 1\}$ for each $k$, in which any two intervals are loosely scale-connected. Each equivalence class also looks like a continuous cluster.
The \emph{maximal interval-merge} proceeds as follows: Merge each $\mathfrak{I}_k$ into a new interval \begin{equation}
\label{maxint}\mathtt{I}^\mathfrak{m}_k:=\langle \mathfrak{a}^\mathfrak{m}_k, \mathfrak{b}^\mathfrak{m}_k\rangle\end{equation} and induce a new scale function $\mathfrak{s}^\mathfrak{m}_k$ on it by means of $\lambda_{\bs}$ as the procedures above Remark~\ref{RM45}. Eventually, the merging of equivalent class $\mathfrak{I}_k$ relative to $\lambda_\bs$ is $(\mathtt{I}^\mathfrak{m}_k, \mathfrak{s}^\mathfrak{m}_k)$.

We shall now prove that under the given scale measure, the Dirichlet space obtained by this operation is the smallest D-subspace of $(\EE,\FF)$.  %  exists and is obtained by this operation.  % the operation of merge each cluster $\mathfrak{I}_k$ into one, as explained in the following definition and called the maximal merge, one of optional merge operations introduced later.  %The proposition after the following definition tells us
This is also why we call it the `maximal' interval-merge.  %is that it produces the smallest D-subspace. % of $(\EE,\FF)$. % while the minimal merge presents the greatest.

 %\begin{definition}
 %For each $k$, let
 %\[
 %	\mathfrak{a}^\mathfrak{m}_k:=\inf\{x\in J: J\in \mathfrak{I}_k\},\quad \mathfrak{b}^\mathfrak{m}_k:=\sup\{x\in J: J\in \mathfrak{I}_k \}.
 %\]
 %Take a fixed point $\mathfrak{e}^\mathfrak{m}_k$ in $(\mathfrak{a}^\mathfrak{m}_k, \mathfrak{b}^\mathfrak{m}_k)$ and let $\mathfrak{s}^\mathfrak{m}_k$  %with $\mathfrak{s}^\mathfrak{m}_k(\mathfrak{e}^\mathfrak{m}_k)=0$
 %be the scale function induced by the measure $\lambda_\bs$ on $(\mathfrak{a}^\mathfrak{m}_k, \mathfrak{b}^\mathfrak{m}_k)$, being zero at $\mathfrak{e}^\mathfrak{m}_k$. Then define
 %\[
 %	\mathtt{I}^\mathfrak{m}_k:=\langle \mathfrak{a}^\mathfrak{m}_k, \mathfrak{b}^\mathfrak{m}_k\rangle,
 %\]
 %where $\mathfrak{a}^\mathfrak{m}_k \in \mathtt{I}^\mathfrak{m}_k$ (resp. $\mathfrak{b}^\mathfrak{m}_k \in \mathtt{I}^\mathfrak{m}_k$) if and only if $\mathfrak{a}^\mathfrak{m}_k+\mathfrak{s}^\mathfrak{m}_k(\mathfrak{a}^\mathfrak{m}_k)>-\infty$ (resp. $\mathfrak{b}^\mathfrak{m}_k+\mathfrak{s}^\mathfrak{m}_k(\mathfrak{b}^\mathfrak{m}_k)<\infty$).   The pair $(\mathtt{I}^\mathfrak{m}_k, \mathfrak{s}^\mathfrak{m}_k)$ is called the maximal merge of $\mathfrak{I}_k$.
 %\end{definition}

\begin{proposition}\label{PRO414}
The set $\{(\mathtt{I}^\mathfrak{m}_k, \mathfrak{s}^\mathfrak{m}_k): k\geq 1\}$ is a class of effective intervals, whose scale measure equals $\lambda_\bs$. Let $(\mathfrak{E}^\mathfrak{m}, \mathfrak{F}^\mathfrak{m})$ be the Dirichlet form given by effective intervals $\{(\mathtt{I}^\mathfrak{m}_k, \mathfrak{s}^\mathfrak{m}_k): k\geq 1\}$. Then $(\mathfrak{E}^\mathfrak{m}, \mathfrak{F}^\mathfrak{m})$ is a D-subspace of $(\EE,\FF)$ on $L^2(\mathbb{R},m)$. Furthermore, if $(\EE',\FF')$ is a D-subspace of $(\EE,\FF)$ with the scale measure $\lambda_\bs$, then $(\mathfrak{E}^\mathfrak{m}, \mathfrak{F}^\mathfrak{m})$ is also a D-subspace of $(\EE',\FF')$.
\end{proposition}
\begin{proof}
Similar to Lemma~\ref{LM46}, we can deduce that $\{\mathtt{I}^\mathfrak{m}_k: k\geq 1\}$ are mutually disjoint. Thus $\{(\mathtt{I}^\mathfrak{m}_k, \mathfrak{s}^\mathfrak{m}_k): k\geq 1\}$ are effective intervals and clearly, its scale measure equals $\lambda_\bs$.

For the rest of assertion, it is enough to verify that $(\mathfrak{E}^\mathfrak{m}, \mathfrak{F}^\mathfrak{m})$ is a D-subspace of $(\EE',\FF')$.  %, which is a D-subspace of $(\EE,\FF)$.
Let $\{I'_m: m\geq 1\}$ be the effective intervals of $(\EE',\FF')$. It suffices to prove that $\{\mathtt{I}^\mathfrak{m}_k: k\geq 1\}$ is coarser than $\{I'_m: m\geq 1\}$ by Theorem~\ref{thm2.1}. Mimicking Remark~\ref{RM32}~(1), set
\[
	\mathcal{I}_m:=\{I_n: I_n\subset I'_m\}.
\]
Since the scale measure of $(\EE',\FF')$ equals $\lambda_\bs$, it follows from Remark~\ref{RM32}~(1) that the intervals in $\mathcal{I}_m$ are loosely scale-connected with each other. Thus $\mathcal{I}_m\subset \mathfrak{I}_k$ for some $k$, which leads to the conclusion $I'_m\subset \mathtt{I}_k$. That completes the proof.
\end{proof}

\subsubsection{Optional interval-merge}
The maximal interval-merge, merging each equivalence class of loose scale-connection, $\mathfrak{J}_k$, into one interval,
gives the minimal D-subspace. However
the set of intervals in $\mathfrak{I}_k$ may be merged more optionally, so that it produces all possible D-subspaces under the same scale measure.

Let us prepare some ingredients to do other interval-merge operation on \eqref{EQ4IKK}. Fix $k\geq 1$.  %Intuitively each cluster $\mathfrak{I}_k$ may be divided and merged almost arbitrarily. Fix an integer $k$.
A point $x\in [\mathfrak{a}^\mathfrak{m}_k, \mathfrak{b}^\mathfrak{m}_k]$ is called a \emph{left} (resp. \emph{right}) \emph{pre-merging point} if either $x=a_n$ (resp. $x=b_n$) for some $I_n=\langle a_n, b_n\rangle\in \mathfrak{I}_k$ or $x\notin \bigcup_{J\in \mathfrak{I}_k}J$.  A pair of points $x$ and $y$, denoted by $\lfloor x, y \rfloor$, is called a \emph{pre-merging pair} of $\mathfrak{I}_k$ if $x$ is a left pre-merging point, $y$ is a right pre-merging point and $x<y$.  We say that two pre-merging pairs $\lfloor x_1, y_1\rfloor$ and $\lfloor x_2, y_2\rfloor$ are disjoint if $$[x_1, y_1]\cap [x_2, y_2]=\emptyset.$$  Given a pre-merging pair $\lfloor x, y \rfloor$, $\lambda_\bs$ induces a scale function $\mathfrak{s}$ on $(x,y)$ and set
\begin{equation}\label{EQ4IXY}
	\mathtt{I}:=\langle x, y\rangle,
\end{equation}
where $x\in \mathtt{I}$ if and only if $x+\mathfrak{s}(x)>-\infty$, $y\in \mathtt{I}$ if and only if $y+\mathfrak{s}(y)<\infty$. In abuse of terminology, we also call $(\mathtt{I}, \mathfrak{s})$ the merging of $\lfloor x, y\rfloor$ relative to $\lambda_\bs$.

\begin{remark}
The points in $(\mathfrak{a}^\mathfrak{m}_k, \mathfrak{b}^\mathfrak{m}_k)\setminus \left(\bigcup_{J\in \mathfrak{I}_k}J\right)$ are both left and right pre-merging points. The left endpoint $\mathfrak{a}^\mathfrak{m}_k$ of $\mathtt{I}^\mathfrak{m}_k$, possibly being $-\infty$, is a left pre-merging point, and the right endpoint $\mathfrak{b}^\mathfrak{m}_k$, possibly being $\infty$, is a right pre-merging point. Notice that the endpoints of intervals in $\mathfrak{I}_k$, except for $\mathfrak{a}^\mathfrak{m}_k$ and $\mathfrak{b}^\mathfrak{m}_k$, are all closed. Similarly, \eqref{EQ4IXY} must be closed when $x, y \in (\mathfrak{a}^\mathfrak{m}_k, \mathfrak{b}^\mathfrak{m}_k)$.
\end{remark}

We are now well prepared to present the \emph{optional interval-merge} on $\mathfrak{I}_k$ as follows. Take at most countable disjoint pre-merging pairs $\{\lfloor \mathfrak{a}^k_{p}, \mathfrak{b}^k_p\rfloor: 1\leq p \leq N_k\}$ ($N_k\leq \infty$) and denote their merging relative to $\lambda_\bs$ by $\{(\mathtt{I}^k_p, \mathfrak{s}^k_p): 1\leq p\leq N_k\}$. Let $\mathfrak{I}_k^p$ be the class of all the intervals $J\in \mathfrak{I}_k$ such that $J\subset \mathtt{I}^k_p$. Then the set
\begin{equation}\label{EQ4INS}
\left\{(I_n,\bs_n): I_n\in \mathfrak{I}_k \setminus \bigcup_{1\leq p\leq N_k} \mathfrak{I}_p^k\right\} \bigcup \{(\mathtt{I}^k_p, \mathfrak{s}^k_p): 1\leq p\leq N_k\}
\end{equation}
is what we obtain by an interval-merge operation on $\mathfrak{I}_k$, which is called an \emph{optional interval-merge}.
 %Actually, an optional merge takes some groups of intervals to merge, and leaves other intervals unchanged.

 Intuitively speaking, to do an optional interval-merge, we choose pre-merging pairs, then merge the intervals lying inside each pair into a new effective interval, and leave the intervals outside unchanged. In that way, we make effective intervals coarser. Particularly, the maximal interval-merge takes only one pre-merging pair $\lfloor \mathfrak{a}^\mathfrak{m}_k, \mathfrak{b}^\mathfrak{m}_k\rfloor$ and merges all the intervals in $\mathfrak{I}_k$ into $(\mathtt{I}^\mathfrak{m}_k, \mathfrak{s}^\mathfrak{m}_k)$.

\begin{definition}
Let $\{\mathfrak{I}_k: k\geq 1\}$ be the set of equivalence classes of $\{(I_n, \bs_n): n\geq 1\}$ induced by the loose scale-connection with respect to $\lambda_\bs$. An optional interval-merge operation on $\{(I_n, \bs_n): n\geq 1\}$ is to execute an optional interval-merge on each $\mathfrak{I}_k$ for any $k\geq 1$.
\end{definition}

The following proposition indicates that an optional interval-merge operation leads to a D-subspace.

\begin{proposition}\label{PRO417}
The set obtained by an optional interval-merge operation,
\begin{equation}\label{EQ4KIS}
	\bigcup_{k\geq 1}\left(\left\{(I_n,\bs_n): I_n\in \mathfrak{I}_k \setminus \bigcup_{1\leq p\leq N_k} \mathfrak{I}_p^k\right\} \bigcup \{(\mathtt{I}^k_p, \mathfrak{s}^k_p): 1\leq p\leq N_k\}\right)
\end{equation}
is a class of effective intervals, with scale measure $\lambda_\bs$. Furthermore, its associated Dirichlet form $(\mathfrak{E}, \mathfrak{F})$ is a D-subspace of $(\EE,\FF)$ on $L^2(\mathbb{R},m)$.
\end{proposition}

\begin{proof}
For the first assertion, we only need to show the intervals in \eqref{EQ4KIS} are mutually disjoint. Denote all the intervals in \eqref{EQ4INS} by $\mathcal{A}_k$. Clearly, from the second condition in Definition~\ref{DEF42}, we  conclude that $J_1\in\mathcal{A}_k$ and $J_2\in \mathcal{A}_{k'}$ with $k\neq k'$ are disjoint. Now fix an integer $k$, and take $I_n, \mathtt{I}^k_p\in \mathcal{A}_k$. Suppose $a_n< \mathfrak{a}^k_p$ and $I_n\cap \mathtt{I}^k_p\neq \emptyset$. This implies that $b_n=\mathfrak{a}^k_p\in I_n\cap \mathtt{I}^k_p$. Since $\{I_n: n\geq 1\}$ are mutually disjoint, it follows from the definition of left pre-merging point that $\mathfrak{a}^k_p\notin \bigcup_{J\in \mathfrak{I}_k}J$, which contradicts $\mathfrak{a}^k_p\in I_n$.
For the second assertion, since the scale measure associated to \eqref{EQ4KIS} is equal to $\lambda_\bs$, it suffices to prove that \eqref{EQ4KIS} is coarser than $\{I_n:n\geq 1\}$. This fact is clear from the definition of \eqref{EQ4KIS}. That completes the proof.
\end{proof}

An example below is to explain the optional interval-merge operations on the Dirichlet form in Example~\ref{EXA33}. Recall the $\bs(x)=x-e$ on any interval $J$ is the natural scale function on $J$.

\begin{example}
Let $(\EE,\FF)$ be the Dirichlet form in Example~\ref{EXA33} with effective intervals $\{I_n: n\geq 1\}$ and natural scale function on each interval. Note that its scale measure is the Lebesgue measure on $\mathbb{R}$. Since the Cantor set $K$ is nowhere dense, $\{I_n: n\geq 1\}$ are loosely scale-connected to each other and thus there is only one equivalence class $\mathfrak{I}_1$ induced by the loose scale-connection.

In the following, we shall present several examples of pre-merging pairs. The maximal interval-merge corresponds to the pre-merging pair $\lfloor -\infty, \infty\rfloor$ ($N_1=1$). It merges all the intervals into the new one associated with the 1-dim Brownian motion. A trivial optional interval-merge operation is as follows: Let $N_1=\infty$ and set
\[
	\lfloor \mathfrak{a}^1_p, \mathfrak{b}^1_p\rfloor:= \lfloor a_p, b_p\rfloor,\quad p\geq 1.
\]
By this operation, effective intervals remain the same.

The pre-merging pair $\lfloor -\infty, b_n\rfloor$ for some $n\geq 1$ corresponds to operation, which merges all the intervals between $-\infty$ and $b_n$ into the new interval $(-\infty, b_n]$ with the natural scale function on it. Note that $b_n$ is not a left pre-merging point by the definition.  %This imposition could be understood as follows:
In fact, once we take a pre-merging pair $\lfloor b_n, y\rfloor$, its merging must be $[b_n, y\rangle$ with natural scale function on it. Then $[b_n, y\rangle \cap I_n\neq \emptyset$ and thus breaks (E1).

Finally, if we take a point $y\in K\setminus \{0,1, a_n, b_n: n\geq 3\}$, $\lfloor -\infty, y\rfloor$ is also a pre-merging pair. Clearly, there exists a subsequence $\{I_{n_m}:m\geq 1\}$ such that both the endpoints $a_{n_m}\uparrow y$ and $b_{n_m}\uparrow y$ as $m\rightarrow \infty$. Thus the merging of $\lfloor -\infty, y\rfloor$ is $(-\infty, y]$ with the natural scale function on it. Furthermore, since the pre-merging pairs in an optional interval-merge are required to be disjoint, we know that $y$ cannot be either left or right pre-merging point in another pre-merging pair of the same optional interval-merge.
\end{example}

\subsection{Road map to D-subspaces}

We have already introduced two kinds of operations on effective intervals to obtain D-subspaces. Keep in mind that the scale-shrink operation essentially identifies a new scale measure, and the optional interval-merge operation does not change the scale measure. The main result of this section is the following theorem, which states that every D-subspace of $(\EE,\FF)$ can be obtained by firstly a scale-shrink operation to fix a scale measure and then an optional interval-merge operation to acquire wanted effective intervals.
 %These two steps mean that by a scale-shrink operation, we obtain a new class of effective intervals, then we proceed an optional merge operation on the new class and the final class of effective intervals are attained.

\begin{theorem}\label{THM419}
Let $(\EE,\FF)$ and $(\mathfrak{E},\mathfrak{F})$ be two regular and strongly local Dirichlet forms on $L^2(\mathbb{R},m)$ having effective intervals $\{(I_n,\bs_n):n\geq 1\}$ and $\{(\mathtt{I}_k, \mathfrak{s}_k): k\geq 1\}$ respectively. Then $(\mathfrak{E}, \mathfrak{F})$ is a D-subspace of $(\EE,\FF)$ on $L^2(\mathbb{R},m)$ if and only if $\{(\mathtt{I}_k, \mathfrak{s}_k): k\geq 1\}$ is attained from $\{(I_n,\bs_n): n\geq 1\}$ by firstly a scale-shrink operation and then an optional interval-merge operation.
\end{theorem}
\begin{proof}
The sufficiency follows from Propositions \ref{PRO48} and \ref{PRO417}. It suffices to prove the necessity. Recall that $I_n=\langle a_n, b_n\rangle$ and $\mathtt{I}_k=\langle \mathfrak{a}_k, \mathfrak{b}_k\rangle$. Denote the scale measure associated to $(\mathfrak{E},\mathfrak{F})$ by $\lambda_\mathfrak{s}$. Since $\{\mathtt{I}_k: k\geq 1\}$ is coarser than $\{I_n: n\geq 1\}$ by Theorem~\ref{thm2.1}, $\{I_n: n\geq 1\}$ may be divided into classes:
\[
	\mathfrak{I}_k:=\{I_n: I_n\subset \mathtt{I}_k, n\geq 1 \}, \ k\ge 1.
\]
We then proceed scale-shrink operation to make the scale measure be $\lambda_\mathfrak{s}$.
To execute the key step and minimal interval-merge of scale-shrink operation as stated in Definition~\ref{DEF48}, we will
obtain a class of effective intervals. More precisely, set $\bar{\bs}_n$ to be the scale function on $I_n$ induced by $\lambda_{\mathfrak{s}}$.  %Clearly, \eqref{EQ4DSN} holds and the scale measure associated with $\{\bar{\bs}_n: n\geq 1\}$ equals $\lambda_\mathfrak{s}$.
 %Then the minimal merge proceeds as follows:
Divide $\{I_n: n\geq 1\}$ into several equivalence classes induced by the tight scale-connection
 \[	\{\mathfrak{I}_p^1: p\geq 1\},\]
and then merge each $\mathfrak{I}^1_p$ into a new interval $\mathtt{I}^1_p$ with the scale function $\mathfrak{s}^1_p$ induced by $\lambda_\mathfrak{s}$ on $\mathtt{I}^1_p$. Thus we attain a new class of effective intervals
$\{(\mathtt{I}^1_p, \mathfrak{s}^1_p): p\geq 1\}$ with scale measure $\lambda_\mathfrak{s}$.
  %by a scale-shrink operation.

We now prepare to do an optional interval-merge on $\{(\mathtt{I}^1_p, \mathfrak{s}^1_p): p\geq 1\}$, which amounts
to picking a class of pre-merging pairs.
 %Notice that the scale measure of $\{(\mathtt{I}^1_p, \mathfrak{s}^1_p): p\geq 1\}$ is also equal to $\lambda_{\mathfrak{s}}$.
At first $\{\mathtt{I}^1_p: p\geq 1\}$ can be divided into equivalence classes induced by the loose scale-connection with respect to $\lambda_\mathfrak{s}$
\[
\{\mathcal{I}_q: q\geq 1 \},
\]
and define for each $q\geq 1$,
\[
	\mathfrak{I}^2_q:=\bigcup_{\mathtt{I}^1_p\in \mathcal{I}_q} \mathfrak{I}_p^1,
\]
which is a subset of $\{I_n:n\ge 1\}$.
It follows from Propositions~\ref{PRO410} and \ref{PRO414} that $\{\mathfrak{I}^1_p: p\geq 1\}$ is finer than $\{\mathfrak{I}_k: k\geq 1\}$ and $\{\mathfrak{I}_k:k\geq 1\}$ is finer than $\{\mathfrak{I}^2_q: q\geq 1\}$, in other words, for any $p\geq 1$, $\mathfrak{I}^1_p\subset \mathfrak{I}_k$ for some $k$, and for any $k\geq 1$, $\mathfrak{I}_k\subset \mathfrak{I}_q^2$ for some $q$.  We assert now that
\[
\{\lfloor \mathfrak{a}_k, \mathfrak{b}_k\rfloor: \mathfrak{I}_k\subset \mathfrak{I}^2_q \}
\]
is a class of disjoint pre-merging pairs of $\mathcal{I}_q$, which implies that $\{(\mathtt{I}_k,\mathfrak{s}_k):k\geq 1\}$ is attained from $\{(\mathtt{I}^1_p,\mathfrak{s}^1_p): p\geq 1\}$ by an optional interval-merge operation. In fact, we only need to show $\mathfrak{a}_k$ (resp. $\mathfrak{b}_k$) with $\mathfrak{I}_k\subset \mathfrak{I}^2_q$ is a left (resp. right) pre-merging point of $\mathcal{I}_q$. Note that
\[
	\mathfrak{a}_k = \inf\{x\in J: J\in \mathfrak{I}_k \}\geq \inf\{x\in J: J\in \mathfrak{I}^2_q\}=\inf\{x\in J: J\in \mathcal{I}_q\}
\]
and similarly, $\mathfrak{b}_k\leq \sup\{x\in J: J\in \mathcal{I}_q\}$.
When $\mathfrak{a}_k=\inf\{x\in J: J\in \mathcal{I}_q\}$, clearly $\mathfrak{a}_k$ is a left pre-merging point of $\mathcal{I}_q$. Now assume $\mathfrak{a}_k>\inf\{x\in J: J\in \mathcal{I}_q\}$, which implies  that $\mathfrak{a}_k\in\mathtt{I}_k$. Suppose $\mathfrak{a}_k\in \mathtt{I}^1_p:=\langle \mathfrak{a}^1_p, \mathfrak{b}^1_p\rangle$ but $\mathfrak{a}_k\neq \mathfrak{a}^1_p$. This indicates that $\mathfrak{a}_k>\mathfrak{a}^1_p$. Since $\mathtt{I}^1_p\subset \mathtt{I}_k$ or $\mathtt{I}^1_p\cap \mathtt{I}_k=\emptyset$ by Theorem~\ref{thm2.1}, it follows from $\mathfrak{a}_k\in \mathtt{I}_k\cap \mathtt{I}^1_p$ that $\mathtt{I}^1_p\subset \mathtt{I}_k$, which contradicts the fact $\mathfrak{a}_k>\mathfrak{a}^1_p$. Thus we conclude that $\mathfrak{a}_k$ is a left pre-merging point of $\mathcal{I}_q$. Similarly, we can deduce that $\mathfrak{b}_k$ is a right pre-merging point of $\mathcal{I}_q$. That completes the proof.
\end{proof}

We shall end this section by an interesting observation. All the D-subspaces of $(\EE,\FF)$ can be classified by possible scale measures, which are identified in the key step of scale-shrink operation. Let $\lambda_\mathfrak{s}$ be a scale measure satisfying \eqref{EQ3DLS} and $\mathscr{S}(\lambda_\mathfrak{s})$ the class of all the D-subspaces with the scale measure $\lambda_\mathfrak{s}$. Then the minimal interval-merge in Proposition~\ref{PRO410} corresponds to the largest D-subspace in $\mathscr{S}(\lambda_\mathfrak{s})$ and the maximal interval-merge in Proposition~\ref{PRO414} gives the smallest one. Optional interval-merge operations produce all other D-subspaces between them.

\section{D-subspaces generated by a class of functions}\label{SEC5}

In previous sections, we describe how we can obtain all possible D-subspaces from two operations on effective intervals. In this section, we shall come back to analyze how to construct a D-subspace from a particular function
and how it relates to the operations in the previous sections.

Fix a Dirichlet form $(\EE,\FF)$ with the effective intervals $\{(I_n, \bs_n):n\geq 1\}$ and scale measure $\lambda_\ss$.
Recall that $\mathbf{S}(\mathbb{R})$ is the family of all scale functions  on $\mathbb{R}$, i.e.
\[
\mathbf{S}(\mathbb{R})=\{\ff:\mathbb{R}\rightarrow \mathbb{R}\; |\; \ff\text{ is strictly increasing and continuous, }\ff(0)=0\}.
\]
Let $\ff\in\mathbf{S}(\mathbb{R})$ and denote its induced Radon measure by $\lambda_\ff$. Write
\begin{equation}\label{EQ5LFE}
\lambda_\ff^\mathrm{e}:= \lambda_\ff|_{\bigcup_{n\geq 1} I_n},\quad \lambda_\ff^\t:=\lambda_\ff|_{\left( \bigcup_{n\geq 1}I_n\right)^c}.
\end{equation}
(The superscript `$\e$' stands for `effective part' and `$\t$' stands for `trivial part'.)
Set $\ff(\mathbb{R}):=\{\ff(x):x\in \mathbb{R}\}$ and
\[
	\mathscr{C}_\ff:=C_c^\infty\circ \ff=\left\{\varphi\circ \ff: \varphi\in C_c^\infty(\ff(\mathbb{R}))\right\}.
\]
We shall impose the assumption $\mathscr{C}_\ff\subset \FF$ henceforth. Denote the $\EE_1$-closure of $\CC_\ff$ by $\mathfrak{F}$ and define
\[
	\mathfrak{E}(u,v):=\EE(u,v),\quad u,v\in \mathfrak{F}.
\]
Then it is easy to check that $(\mathfrak{E},\mathfrak{F})$ is a D-subspace of $(\EE,\FF)$.
This section is devoted to study this D-subspace.  Note that the special case with $\ff(x)=x$, the natural scale function, has been studied in \cite[\S3]{LY17}.

\subsection{Basic assumption}

The following lemma brings into play \eqref{EQ5LFE} and characterizes the basic assumption $\mathscr{C}_\ff\subset \FF$.

\begin{lemma}\label{LM51}
The condition $\mathscr{C}_\ff\subset \FF$ is satisfied if and only if
\begin{equation}\label{EQ5LEF}
	\lambda^\e_\ff\ll \lambda_\bs,\quad \frac{d\lambda^\e_\ff}{d\lambda_\bs}\in L^2_\mathrm{loc}(\mathbb{R}, \lambda_\bs).
\end{equation}
\end{lemma}
\begin{proof}
For the sufficiency, on account of $\mathscr{C}_\ff\subset L^2(\mathbb{R})$ it suffices to show $\varphi\circ \ff\in \FF$ and $\EE(\varphi\circ \ff, \varphi\circ \ff)<\infty$ for any $\varphi\in C_c^\infty(\ff(\mathbb{R}))$. Indeed, $\varphi\circ \ff|_{I_n} \ll \ss_n$ by $\lambda^\e_\ff\ll \lambda_\bs$, and
\[
\begin{aligned}
	2\EE(\varphi\circ \ff, \varphi\circ \ff)&=\sum_{n\geq 1} \int_{I_n} \left(\varphi'\circ \ff\right)^2\left(\frac{d\ff}{d\ss_n}\right)^2d\ss_n \\
	&=\int_\mathbb{R} \left(\varphi'\circ \ff\right)^2\left( \frac{d\lambda^\e_\ff}{d\lambda_\bs}\right)^2d\lambda_\bs.
\end{aligned}
\]
Then \eqref{EQ5LEF} implies $\EE(\varphi\circ \ff, \varphi\circ \ff)<\infty$.

To the contrary, we need only to note $\mathscr{C}_\ff\subset \FF$ implies that $\ff\in \FF_\mathrm{loc}$ and then \eqref{EQ5LEF} follows from the expression of $(\EE,\FF)$. That completes the proof.
\end{proof}

\subsection{Scale measure and optional interval-merge}\label{SEC52}

Clearly, $(\mathfrak{E}, \mathfrak{F})$ is a regular and strongly local Dirichlet form on $L^2(\mathbb{R},m)$. Thus it can be represented by another class of effective intervals $$\{(\mathtt{I}_k, \mathfrak{s}_k): k\geq 1\}.$$ Its scale measure is denoted by $\lambda_\mathfrak{s}$ as before. It is said in Theorem~\ref{THM419} that $\{(\mathtt{I}_k, \mathfrak{s}_k): k\geq 1\}$ is derived from $\{(I_n, \bs_n):n\geq 1\}$ by firstly a scale-shrink operation and then an optional interval-merge. In this subsection, we shall identify the expected scale measure $\lambda_\s$ and describe briefly the optional interval-merge to attain $(\mathfrak{E}, \mathfrak{F})$.  % by means of $\ff$.
The proof is postponed to Theorem~\ref{THM56} in next subsection.  %Note incidentally that we ignore minimal interval-merge in the scale-shrink operation, since it contains the fewest merging, and could be covered by optional interval-merge (see Remark~\ref{RM54} below).

We shall first formulate the scale measure $\lambda_\mathfrak{s}$. Since $\lambda_\ss$ and $\lambda^\e_\ff$ are $\sigma$-finite on $\mathbb{R}$, we have the following Lebesgue-Radon-Nikodym decomposition:
\begin{equation}\label{EQ5LSR}
	\lambda_\ss=\bar{\lambda}+\kappa=g\cdot \lambda^\e_\ff +\kappa,
\end{equation}
where $\kappa$ is singular with respect to  $\lambda^\e_\ff$. The crucial fact we will prove later is that $\lambda_\s$ coincides with the absolute part, i.e.
\begin{equation}\label{EQ5LSL}
	\lambda_\s=\bar{\lambda}\quad (=g\cdot \lambda^\e_\ff).
\end{equation}
Before moving on to optional interval-merge, let us explain the easy part, why $\bar{\lambda}$ is a shrinking of $\lambda_\ss$, or it actually induces a proper scale function on each $I_n$.  %These scale functions constitute the ingredients of key step in scale-shrink operation.

\begin{lemma}\label{LM52}
Let $\bar{\lambda}$ be in \eqref{EQ5LSR}. Then for each $n\geq 1$, $\bar{\lambda}|_{I_n}$ induces a scale function $\bar{\ss}_n\in \mathbf{S}(I_n)$ satisfying \eqref{EQ4DSN}.
\end{lemma}
\begin{proof}
Clearly, $\bar{\lambda}\ll \lambda_\ss$, $$\frac{d\bar{\lambda}}{d\lambda_\ss}=0\ \text{or}\ 1, $$ $\lambda_\ss$-a.e. and $\bar{\lambda}(\{x\})=0$ for any $x\in \bigcup_{n\geq 1}I_n$. It suffices to show that $\bar{\lambda}$ is fully supported on $I_n$. Since $\lambda^\e_\ff$ is fully supported on $I_n$, this amounts to $g>0$, $\lambda^\e_\ff$-a.e. Let $H$ be a measurable subset of $\mathbb{R}$ such that $\kappa(H)=\lambda^\e_\ff(H^c)=0$. Write
\[
	Z_g:=\{x\in H: g(x)=0\}.
\]
Then $\lambda_\ss(Z_g)=\bar{\lambda}(Z_g)=0$, and it follows from \eqref{EQ5LEF} that $\lambda^\e_\ff(Z_g)=0$. That completes the proof.
\end{proof}

We now move to optional interval-merge, which depends on the equivalence classes induced by  $\ff$. %-scale-connection below.

\begin{definition}\label{DEF53}
We say that $I_i$ and $I_j$ are $\ff$-scale-connected, or $I_i$ is $\ff$-scale-connected to $I_j$, with respect to the scale measure $\bar{\lambda}$,
if
\begin{itemize}
\item[(1)] $\lambda^\t_\ff([e_i,e_j])=0$, where $\lambda^\t_\ff$ is given by \eqref{EQ5LFE};
\item[(2)] $\bar{\lambda}([e_i,e_j])<\infty$.
\end{itemize}
\end{definition}
\begin{remark}\label{RM54}
In the case of $\ff$ being the natural scale function, the simpler terminology `scale-connection' was used in \cite[Definition~3.5]{LY17}  instead.
It is seen that the second condition in Definition~\ref{DEF53} is related to the scale measure, just as in the definition of tight scale-connection and loose scale-connection.
Needless to say, under the same scale measure, tight scale-connection implies $\ff$-scale-connection, and $\ff$-scale-connection implies loose scale-connection. %As a consequence, the minimal interval-merge could be actually covered by the merging induced by $\ff$ as below.
\end{remark}

Clearly, $\ff$-scale-connection is also an equivalence relation on $\{I_n:n\geq 1\}$. Thus these intervals may be divided into equivalence classes as usual:
\[
	\{\bar{\mathfrak{I}}_k: k\geq 1\},
\]
where $\bar{\mathfrak{I}}_k\subset \{I_n:n\geq 1\}$, in which the intervals are mutually $\ff$-scale-connected. Then we merge each group $\bar{\mathfrak{I}}_k$ under the scale measure $\bar{\lambda}$ just as the procedures above Remark~\ref{RM45} into an interval $\bar{\mathtt{I}}_k$ and {obtain the merging  $(\bar{\mathtt{I}}_k, \bar{\s}_k)$ of $\bar{\mathfrak{I}}_k$. We call this operation the $\ff$-interval-merge.} % which

Mimicking Lemma~\ref{LM46}, we have the following lemma. The proof is analogical and omitted.

\begin{lemma}\label{LM55}
The sequence $\{(\bar{\mathtt{I}}_k, \bar{\s}_k): k\geq 1\}$ is a class of effective intervals with the scale measure $\bar{\lambda}$.
\end{lemma}

Not surprisingly, we shall conclude
\begin{equation}\label{EQ5IKS}
	\{(\mathtt{I}_k, \mathfrak{s}_k): k\geq 1\}=\{(\bar{\mathtt{I}}_k, \bar{\s}_k): k\geq 1\}.
\end{equation}
That is to say, $\ff$-interval-merge gives the D-subspace $(\mathfrak{E},\mathfrak{F})$ directly,  %(under the scale measure $\bar{\lambda}$) is directly presented by $\ff$-scale-connection in Definition~\ref{DEF53}
without seeking the minimal interval-merge. This means that $\ff$-interval-merge combines the minimal interval-merge and an optional interval-merge together.
It is worth noting and not hard to see that  $\ff$-interval-merge is
\begin{itemize}
\item[(a)] identified with the minimal interval-merge, if and only if once $\bar{\lambda}([e_i,e_j])<\infty$ for $i\neq j$,
\begin{equation}\label{EQ5LTF}
	\lambda^\t_\ff\left([e_i,e_j]\right)=0\Rightarrow[e_i,e_j]\setminus \left(\bigcup_{n\geq 1}I_n\right) \text{ is of at most countable points};
\end{equation}
\item[(b)] identified with the maximal interval-merge, if and only if once $\bar{\lambda}([e_i,e_j])<\infty$ for $i\neq j$,
\begin{equation}\label{EQ5EIE}
[e_i,e_j]\setminus \left(\bigcup_{n\geq 1}I_n\right) \text{ is nowhere dense} \Rightarrow \lambda^\t_\ff\left([e_i,e_j]\right)=0.
\end{equation}
\end{itemize}

\subsection{D-subspace generated by $\CC_\ff$}

We are now in a position to phrase and prove the main theorem of this section.

\begin{theorem}\label{THM56}
Assume that $\CC_\ff\subset \FF$, equivalently \eqref{EQ5LEF} holds. Let $(\mathfrak{E},\mathfrak{F})$ be the D-subspace of $(\EE,\FF)$ generated by $\CC_\ff$, and $(\bar{\mathfrak{E}},\bar{\mathfrak{F}})$ the Dirichlet form represented by effective intervals $\{(\bar{\mathtt{I}}_k, \bar{\s}_k): k\geq 1\}$ in Lemma~\ref{LM55}. Then
\[
	(\mathfrak{E},\mathfrak{F})=(\bar{\mathfrak{E}},\bar{\mathfrak{F}}).
\]
\end{theorem}
\begin{proof}
We first prove $\CC\subset \bar{\mathfrak{F}}$ by applying Lemma~\ref{LM51}. Indeed, note that the scale measure associated to $(\bar{\mathfrak{E}},\bar{\mathfrak{F}})$ is $\bar{\lambda}=g\cdot \lambda^\e_\ff$, and $g>0$, $\lambda^\e_\ff$-a.e. by Lemma~\ref{LM52}. Recall that $H$ is the measurable subset of $\mathbb{R}$ in the proof of Lemma~\ref{LM52}. Then $\lambda^\e_\ff\ll \bar{\lambda}$ and
\[
\left(\frac{d\lambda^\e_\ff}{d\bar{\lambda}}\right)^2d\bar{\lambda}=\frac{1}{g^2}\cdot 1_Hd\lambda_\ss=\left(\frac{d\lambda^\e_\ff}{d\lambda_\ss}\bigg|_H \right)^2\cdot 1_Hd\lambda_\ss=\left(\frac{d\lambda^\e_\ff}{d\lambda_\ss} \right)^2d\lambda_\ss\bigg|_H,
\]
on account of $$\bar{\lambda}=1_H\cdot \lambda_\ss\ \text{and}\ g=\frac{d\bar{\lambda}}{d\lambda^\e_\ff}=1_H\cdot\frac{d\lambda_\ss}{d\lambda^\e_\ff}.$$ It follows from \eqref{EQ5LEF} that $d\lambda^\e_\ff/d\bar{\lambda}\in L^2_\mathrm{loc}(\mathbb{R}, \bar{\lambda})$. Hence by Lemma~\ref{LM51}, it follows that $\CC_\ff\subset \bar{\mathfrak{F}}$. In addition, for any $u\in \CC_\ff\subset \FF\cap \bar{\mathfrak{F}}$, we have $u\ll \lambda_\ss, u\ll \bar{\lambda}$ and
\[
\EE(u,u)=\frac{1}{2}\int_\mathbb{R} \left(\frac{du}{d\lambda_\ss}\right)^2d\lambda_\ss=\frac{1}{2}\int_\mathbb{R}\left(\frac{du}{d\bar{\lambda}}\cdot \frac{d\bar{\lambda}}{d\lambda_\ss}\right)^2d\lambda_\ss.
\]
Since $$\left(\frac{d\bar{\lambda}}{d\lambda_\ss}\right)^2d\lambda_\ss=\left(1_H\right)^2d\lambda_\ss=d\bar{\lambda},$$ it hold that
\begin{equation}\label{EQ5EUU}
	\EE(u,u)=\frac{1}{2}\int_\mathbb{R}\left( \frac{du}{d\bar{\lambda}}\right)^2d\bar{\lambda}=\bar{\mathfrak{E}}(u,u).
\end{equation}
Therefore, we conclude that $\mathfrak{F}\subset \bar{\mathfrak{F}}$ and $\bar{\mathfrak{E}}|_{\mathfrak{F}\times \mathfrak{F}}=\mathfrak{E}$.

Next, we prove \eqref{EQ5LSL}.   Since $(\mathfrak{E},\mathfrak{F})$ is a D-subspace of $(\bar{\mathfrak{E}},\bar{\mathfrak{F}})$, Theorem~\ref{thm2.1} tells that $\{(\mathtt{I}_k,\mathfrak{s}_k)\}$ is coarser than $\{(\bar{\mathtt{I}}_k,\bar{\mathfrak{s}}_k)\}$. Consider the part Dirichlet forms $(\mathfrak{E}_J,\mathfrak{F}_J)$ and $(\bar{\mathfrak{E}}_J, \bar{\mathfrak{F}}_J)$ of $(\mathfrak{E},\mathfrak{F})$ and $(\bar{\mathfrak{E}},\bar{\mathfrak{F}})$ respectively on $J:=\mathring{\bar{\mathtt{I}}}_k$, i.e. the interior of $\bar{\mathtt{I}}_k$. They are two irreducible Dirichlet forms and  $(\mathfrak{E}_J,\mathfrak{F}_J)$ is a D-subspace of $(\bar{\mathfrak{E}}_J, \bar{\mathfrak{F}}_J)$. Write $\ff_J:=\ff|_J$. Clearly, $C_c^\infty\circ \ff_J$ is a special standard core of $(\mathfrak{E}_J,\mathfrak{F}_J)$. Since
\[
	d\bar{\s}_k=g\cdot \lambda^\e_\ff|_J=g\cdot d\ff_J\ll d\ff_J,
\]
it follows from \cite[Theorem~3.2]{LY17} that  $C_c^\infty\circ \ff_J$ is also dense in $\bar{\mathfrak{F}}_J$. Consequently, $(\mathfrak{E}_J,\mathfrak{F}_J)$ and $(\bar{\mathfrak{E}}_J, \bar{\mathfrak{F}}_J)$ are identical and particularly, $\lambda_\s|_J=\bar{\lambda}|_J$. Note that
$$\lambda_\s\left(\bigcup_{k\geq 1}\mathtt{I}_k\setminus \bigcup_{k\geq 1}\bar{\mathtt{I}}_k\right)=0$$ by Remark~\ref{RM32}. Consequently, $\lambda_\s=\bar{\lambda}$.

Finally, we derive \eqref{EQ5IKS}. By the fact that $(\mathfrak{E},\mathfrak{F})$ is a D-subspace of $(\bar{\mathfrak{E}},\bar{\mathfrak{F}})$ with the same scale measure, it suffices to verify that any $\mathtt{I}_m$ contains only one $\bar{\mathtt{I}}_k$. We shall prove it by contradiction.  % and take such that $\bar{\mathtt{I}}_k\subset \mathtt{I}_m$.
Suppose that another $\bar{\mathtt{I}}_j\subset \mathtt{I}_m$. We assert that $I_p\in \bar{\mathfrak{I}}_k$ is $\ff$-scale-connected to $I_q\in \bar{\mathfrak{I}}_j$, which contradicts the definition of equivalence classes. In fact, $\bar{\lambda}=\lambda_\s$ is a Radon measure on $\mathtt{I}_m$, and hence $\bar{\lambda}([e_p, e_q])<\infty$. On the other hand, $C_c^\infty\circ \ff|_{J}$ is a core of part Dirichlet form $(\mathfrak{E}_J, \mathfrak{F}_J)$ of $(\mathfrak{E},\mathfrak{F})$ on $J:=\mathring{\mathtt{I}}_m$, and $(\mathfrak{E}_J, \mathfrak{F}_J)$ is an irreducible Dirichlet form with the scale function induced by $\bar{\lambda}|_J$. This implies that
\[
	\lambda_\ff|_J\ll \bar{\lambda}|_J. %\ll \lambda_\ss|_J.
\]
Combining $\bar{\lambda}|_J\ll \lambda_\ss|_J$, we have $\lambda_\ff|_J\ll \lambda_\ss|_J$. In particular, $\lambda^\t_\ff|_J=0$ and then
$\lambda^\t_\ff([e_p,e_q])=0$. As a consequence, $I_p$ is $\ff$-scale-connected to $I_q$.

 %Furthermore, we may easily conclude $\{\bar{\mathtt{I}}_k:k\geq 1\}=\{\mathtt{I}_k:k\geq 1\}$ from the fact that  . Then from $\lambda_\s=\bar{\lambda}$,
We have reached eventually the conclusion
\[
	\{(\mathtt{I}_k, \mathfrak{s}_k): k\geq 1\}=\{(\bar{\mathtt{I}}_k, \bar{\s}_k): k\geq 1\}.
\]
That completes the proof.
\end{proof}

A useful corollary of this theorem is as follows.

\begin{corollary}\label{COR57}
 %Under the same settings as Theorem~\ref{THM56},
$(\mathfrak{E},\mathfrak{F})$ (or $(\bar{\mathfrak{E}},\bar{\mathfrak{F}})$) has the same scale measure as $(\EE,\FF)$, if and only if
\begin{equation}\label{EQ5LSLE}
	 \lambda_\ss\ll \lambda^\e_\ff.
\end{equation}
Furthermore, $\CC_\ff$ is a special standard core of $(\EE,\FF)$ if and only if \eqref{EQ5LSLE} and any one of the following assertions hold:
\begin{itemize}
\item[(1)] Each $I_i$ is $\ff$-scale-isolated, in other words, it is not $\ff$-scale-connected to any other interval.
\item[(2)] \eqref{EQ5LTF} holds, i.e. if $\lambda_\ss([e_i,e_j])<\infty$ for some $i\neq j$, then $\lambda^\t_\ff\left([e_i,e_j]\right)=0$ implies $[e_i,e_j]\setminus \left(\bigcup_{n\geq 1}I_n\right)$ is of at most countable points.
\end{itemize}
\end{corollary}
\begin{proof}
Clearly, \eqref{EQ5LSLE} amounts to $\bar{\lambda}=\lambda_\ss$ by \eqref{EQ5LSR}. The first assertion implies the equivalence relation induced by $\ff$-scale-connection is trivial and no interval-merge needs to do. The second assertion means $\ff$-interval-merge actually coincides with the minimal interval-merge. Then the conclusion follows from Remark~\ref{RM45}.
\end{proof}

This corollary provides a simple way to find a `nice' special standard core like $\CC_\ff$ of $(\EE,\FF)$.  In practice, we only need to find $\ff\in \mathbf{S}(\mathbb{R})$ satisfying \eqref{EQ5LEF}, \eqref{EQ5LTF} and \eqref{EQ5LSLE}, and then $\CC_\ff=C_c^\infty\circ \ff$ is an expected special standard core.
 %We can also see from Corollary ~\ref{COR57} that if $\CC_\ff$ satisfying
 %\eqref{EQ5LEF}, then
Many examples where $\ff$ is the natural scale function can be found in \cite{LY17}. We give more examples below, which tell us that $\lambda^\t_\ff$ is very flexible to obtain these cores. This is the reason we use the superscript `$\t$' in $\lambda^\t_\ff$ to stand for `trivial part'. We also highlight that $\ff$ is only a medium to induce this core (or produce a D-subspace), and the measure $\lambda^\e_\ff$ is not necessarily equal to the scale measure of $(\EE,\FF)$ (or $(\mathfrak{E},\mathfrak{F})$).

\begin{example}\label{EXA58}
Let us consider the Dirichlet form $(\EE,\FF)$ in Example~\ref{EXA33}. Further we impose \eqref{EQ5LEF} and \eqref{EQ5LSLE}, for instance, $\lambda^\e_\ff$ is taken to be the Lebesgue measure on $\mathbb{R}$. As a consequence, $\bar{\lambda}=\lambda_\ss$, and
\[
	\bar{\lambda}([e_i,e_j])<\infty,\quad \forall i\neq j.
\]
This indicates $C_c^\infty\circ \ff$ is a core of $(\EE,\FF)$, if and only if
\begin{equation}\label{EQ5LTFEI}
	\lambda^\t_\ff([e_i,e_j])>0,\quad \forall i\neq j.
\end{equation}

 In \cite{LY17}, the authors have shown $C_c^\infty(\mathbb{R})$ is not $\EE_1$-dense in $\FF$. In fact, in this case $\ff(x)=x$. Then $\bar{\lambda}=\lambda_\ss$ is the Lebesgue measure and all the intervals are mutually $\ff$-scale-connected. Accordingly, the closure of $C_c^\infty(\mathbb{R})$ is the Dirichlet form of 1-dim Brownian motion.

Let $\mathfrak{c}$ be the standard Cantor function with $\mathfrak{c}(x)=0$ for any $x\leq 0$, $\mathfrak{c}(x)=1$ for any $x\geq 1$. Set
\begin{equation}\label{EQ5FXX}
\ff(x)=x+\mathfrak{c}(x),\quad x\in \mathbb{R}. 	
\end{equation}
Clearly, $\ff\in \mathbf{S}(\mathbb{R})$. We assert \eqref{EQ5LEF}, \eqref{EQ5LSLE} and \eqref{EQ5LTFEI} hold, and hence $C_c^\infty\circ \ff$ is a core of $(\EE,\FF)$ (Note incidentally $\ff_l(x):=x+l\cdot\mathfrak{c}(x)$ for any $l>0$ also satisfies these conditions).
Indeed, $\lambda^\e_\ff=dx$ and $\lambda^\t_\ff=d\mathfrak{c}$. For any $i\neq j$, clearly $\lambda^\t_\ff([e_i,e_j])=d\mathfrak{c}([e_i,e_j])>0$. This is nothing but \eqref{EQ5LTFEI}.
\end{example}

 \begin{example}\label{EXA59}
 We still consider the intervals in Example~\ref{EXA33}, but replace the scale function on $I_n$ for $n\geq 3$ by
 \[
 		\ss_n(x):=\frac{x-e_n}{|a_n-b_n|}.
 \]
If $\ff$ is taken to satisfy \eqref{EQ5LEF} and \eqref{EQ5LSLE} (such as $\lambda^\e_\ff:=dx|_{\cup_{n\geq 1}I_n}$), then $\bar{\lambda}=\lambda_\ss$ and
\begin{equation}\label{EQ5LEI}
	\bar{\lambda}([e_i,e_j])=\infty,\quad \forall i\neq j.
\end{equation}
This implies the minimal interval-merge, $\ff$-interval-merge and maximal interval-merge are identified under the scale measure $\bar{\lambda}$. Particularly, $C_c^\infty\circ \ff$ is a special standard core of the Dirichlet form produced by them.

In practice, it has been shown in \cite[Example~3.8~(3)]{LY17} that $C_c^\infty(\mathbb{R})$ is a core of this Dirichlet form, in which $\ff$ is the natural scale function. Certainly, any $\ff\in \mathbf{S}(\mathbb{R})$ with $\lambda^\e_\ff=c\cdot dx$ for some constant $c>0$ (such as \eqref{EQ5FXX}, as well as $\ff_l$) also induces a special standard core.
\end{example}

\begin{example}\label{EXA510}
Another similar example of Dirichlet form is presented in \cite[Example~3.8~(4)]{LY17}, in which the standard Cantor set $K$ is replaced by a generalized Cantor set, and $\ss_n$ is still taken to be the natural scale function on each interval. If \eqref{EQ5LEF} and \eqref{EQ5LSLE} hold (such as $\lambda^\e_\ff:=dx|_{\cup_{n\geq 1}I_n}$), then $\bar{\lambda}=\lambda_\ss$ is the Lebesgue measure on $\cup_{n\geq 1}I_n$. Thus
\[
	\bar{\lambda}([e_i,e_j])<\infty,\quad \forall i\neq j.
\]
In this case, $C_c^\infty\circ \ff$ is a core of $(\EE,\FF)$, if and only if \eqref{EQ5LTFEI} holds.

For example, $\lambda^\t_\ff=dx|_{\mathbb{R}\setminus \cup_{n\geq 1}I_n}$ satisfies \eqref{EQ5LTFEI}, and particularly, $C_c^\infty(\mathbb{R})$ is a core of $(\EE,\FF)$. Another example of $\lambda^\t_\ff$ to satisfy \eqref{EQ5LTFEI} is the measure induced by the generalized Cantor function related to $K$. Furthermore, we can also conclude that \eqref{EQ5FXX} with this generalized Cantor function in place of $\mathfrak{c}$ produces another special standard core.
\end{example}

We give a remark about the maximal interval-merge under the scale measure $\bar{\lambda}$. Apparently, this maximal interval-merge produces a Dirichlet form with a special standard core $C_c^\infty\circ \ff$, if and only if it has no difference with the optional interval-merge stated in \S\ref{SEC52}. By the remark after Lemma~\ref{LM55}, this amounts to \eqref{EQ5EIE}.
In Example~\ref{EXA59}, \eqref{EQ5EIE} is always valid since \eqref{EQ5LEI} holds. However in Example~\ref{EXA58} and \ref{EXA510},
\[
	\bar{\lambda}([e_i,e_j])<\infty,\quad \forall i\neq j,
\]
and the intervals have a Cantor-type stucture. Then \eqref{EQ5EIE} means $\lambda^\t_\ff\equiv 0$. For example, set
\[
	\ff(x):=\int_0^x 1_{K^c}(y)dy,\quad x\in \mathbb{R},
\]
where $K$ is the standard Cantor set or a generalized Cantor set. Clearly, \eqref{EQ5EIE} holds for this $\ff$.
As a result, $(\mathfrak{E},\mathfrak{F})$ coincides with the Dirichlet form produced by the maximal interval-merge. In the case of standard Cantor set, $(\mathfrak{E},\mathfrak{F})$ is nothing but 1-dim Brownian motion. Nevertheless, in the case of generalized Cantor set, $(\mathfrak{E},\mathfrak{F})$, associated with an irreducible diffusion on $\mathbb{R}$ which is deeply described in \cite{LY14}, is a proper D-subspace of 1-dim Brownian motion.

\subsection{Existence of special standard core}

We have seen that the closure of $\CC_\ff$ with $\ff\in\mathbf{S}(\mathbb{R})$ satisfying \eqref{EQ5LEF}
is a D-subspace of $(\EE,\FF)$ and how to reach it through a scale-shrink operation and an optional interval-merge operation.
We have also seen that a special standard core of the form $\CC_\ff$ with $\ff\in \mathbf{S}(\mathbb{R})$ exists for the Dirichlet forms in Examples~\ref{EXA58}, \ref{EXA59} and \ref{EXA510}.  %Now we state a general result, which asserts that such a core always exists.
It is then natural to ask if any D-subspace is generated by such an $\ff$. We shall answer this question by a slightly more general result.

\begin{theorem}\label{T58SSC}
Let $(\EE,\FF)$ be the Dirichlet form on $L^2(\mathbb{R},m)$ with the effective intervals $\{(I_n,\ss_n):n\geq 1\}$ as before. Then there exists a function $\ff\in \mathbf{S}(\mathbb{R})$ such that $\CC_\ff=C_c^\infty\circ \ff$ is a special standard core of $(\EE,\FF)$.
\end{theorem}
\begin{proof}
Briefly, by Corollary ~\ref{COR57}, we need to construct $\ff\in \mathbf{S}(\mathbb{R})$ such that $\lambda_\ff^{\e}\simeq\lambda_\ss$ (mutually absolutely continuous or equivalent), and for $i\not=j$,
$I_i$ and $I_j$ are not $\ff$-scale-connected.
For obtaining such an $\ff$, it suffices to construct a fully supported Radon measure $\lambda$ on $\mathbb{R}$, which charges no set of single point and satisfies  \eqref{EQ5LEF}, \eqref{EQ5LTF} and \eqref{EQ5LSLE}, so that $d\ff=\lambda$.

Write $G:=\bigcup_{n\geq 1}\mathring{I}_n$, where $\mathring{I}_n$ is the interior of $I_n$, and $F:=G^c$. We shall construct $\lambda_1=\lambda|_G$ and $\lambda_2=\lambda|_F$ respectively.

The construction of $\lambda_1$ amounts to finding a Radon measure which is equivalent to $\lambda_\ss$. Actually we may find a finite measure $\lambda_1$ fully supported on $G$ such that
\begin{equation}\label{EQ5LLS}
\lambda_1\simeq \lambda_\ss,\quad \frac{d\lambda_1}{d\lambda_\ss}\in L^2_\mathrm{loc}(\mathbb{R},\lambda_\ss),
\end{equation}
where $\simeq$ means that two measures are mutually absolutely continuous.
In fact, for each $n\geq 1$, write $J_n:=\ss_n(\mathring{I}_n)$ and take a strictly positive and continuous function $h_n$ on $J_n$ such that $\int_{J_n}h_n(x)dx\leq 1/n^2$ and $\int_{J_n}h_n(x)^2dx\leq 1/n^2$. Let $\mathtt{t}_n:=\mathtt{s}^{-1}_n$ be the inverse function of $\ss_n$ and set $g_n:=h_n\circ \mathtt{t}_n$ and
\begin{equation}\label{EQ5LNG}
\lambda_1:=\sum_{n\geq 1}g_n\cdot d\ss_n.
\end{equation}
We verify that $\lambda_1$ is such a measure. It is finite since
\[
\lambda_1(\mathbb{R})=\sum_{n\geq 1}\int_{J_n}h_n(x)dx\leq \sum_{n\geq 1}\frac{1}{n^2}<\infty.
\]
It is clear that $\lambda_1\simeq \lambda_\ss$ by the definition and the fact that $g_n$ is strictly positive.
Since
\[
\int_\mathbb{R} \left(\frac{d\lambda_1}{d\lambda_\ss} \right)^2d\lambda_\ss=\sum_{n\geq 1} \int_{\mathring{I}_n}g_n^2d\ss_n=\sum_{n\geq 1}\int_{J_n}h^2_n(x)dx\leq \sum_{n\geq 1}\frac{1}{n^2}<\infty,
\]
we have ${d\lambda_1/d\lambda_\ss}\in L^2(\mathbb{R},\lambda_\ss)$.
 %Moreover, the conclusions that $\lambda_1$ has full support on $G$ and $\lambda_\ss\ll \lambda_1$ are implied by the fact that $g_n$ is strictly positive.

The role of $\lambda_2$ is to separate $I_i$ and $I_j$ when $\lambda_\ss([e_i,e_j])<\infty$, i.e.,
$\lambda_2([e_i,e_j])=0$ if and only if $[e_i,e_j]\setminus (\bigcup_n I_n)$ is countable. We construct the measure $\lambda_2$ supported on $F$ in the following manner. Write $F=\mathring{F}\cup \partial F$, where $\mathring{F}$ is the interior of $F$ and $\partial F=F\setminus \mathring{F}$, which is a nowhere dense closed set.
 %Define the measure on $\mathring{F}$,
 %\begin{equation}\label{EQ5LFX}
 %	\lambda_2|_{\mathring{F}}:=dx|_{\mathring{F}},
 %\end{equation}
 %i.e. the Lebesgue measure restricted to $\mathring{F}$.
 %Let $K=(\partial F)'$, the set of limit points of $\partial F$. Then $\partial F\setminus K$ is the set of
 %isolated points of $\partial F$, which is at most countable. If $K$ is not empty, it must be a set of Cantor-type,
 %i.e., a nowhere dense perfect set.

Since $\mathring{F}$ is open, it consists of at most countable disjoint open intervals. Denote these intervals by $\{I^{\mathring{F}}_k: k\geq 1\}$ and set
\[
\mathscr{I}:=\{\mathring{I}_n: n\geq 1\}\cup \{I^{\mathring{F}}_k: k\geq 1\}.
\]
A relation `$\sim$' on the intervals in $\mathscr{I}$ is defined as follows: for $J_1,J_2\in \mathscr{I}$, $J_1\sim J_2$ if and only if $[e^1,e^2]\cap \partial F$ is of at most countable points, where $e^1, e^2$ are two (arbitrary) points in $J_1$ and $J_2$ respectively. Clearly, `$\sim$' is an equivalence relation on $\mathscr{I}$, and we denote all the equivalence classes by
\[
	\bigcup_{m\geq 1} \mathcal{I}_m,
\]
where $\{\mathcal{I}_m: m\geq 1\}$ are disjoint subsets of $\mathscr{I}$. For each $m\geq 1$, we merge the intervals in $\mathcal{I}_m$ into a new open interval. More precisely, set $\alpha_m:=\inf\{x\in J: J\in \mathcal{I}_m\}$ and $\beta_m:=\sup\{x\in J: J\in \mathcal{I}_m\}$. Then $(\alpha_m, \beta_m)$ is the interval we obtained. By this operation, we attain a family of at most countable intervals
\[
\mathscr{I}^\sim=\{(\alpha_m,\beta_m): m\geq 1\}.
\]
Let $$K:=\left(\bigcup_{m\ge 1} (\alpha_m,\beta_m)\right)^c.$$ Then $K$ is a closed subset of $\partial F$ and $\partial F\setminus K$ is at most countable. Furthermore $K$ has no isolated points because the number of isolated points is at most countable. Hence if $K$ is non-empty, it is a set of Cantor-type, i.e., a nowhere dense perfect set.
It is well-known that there exists a Cantor function $c_K$ on $\mathbb{R}$, continuous and increasing, so that
the induced measure $dc_K$ is fully supported on $K$, i.e.,
$$\text{supp}(dc_K)=K.$$

 %=======================================================================================================

 %If $\mathscr{I}^\sim$ contains only one interval, we set
Now $F=\mathring{F}\cup K\cup(\partial F\setminus K)$, and we define
\begin{equation}\label{EQ5LPF}
	\lambda_2:=1_{\mathring{F}}\cdot dx+dc_K,
\end{equation}
which assigns the Lebesgue measure on $\mathring{F}$, the measure induced by Cantor function $c_K$ on $K$ and zero on $\partial F\setminus K$, an at most countable set.

Finally, we check the measure $$\lambda:=\lambda_1+\lambda_2,$$  %where $\lambda_1$ is \eqref{EQ5LNG} and $\lambda_2$ is the sum of $\lambda_2|_{\mathring{F}}$ in \eqref{EQ5LFX} and $\lambda_2|_{\partial F}$ in \eqref{EQ5LPF} or \eqref{EQ5LFI},
satisfies \eqref{EQ5LEF}, \eqref{EQ5LTF} and \eqref{EQ5LSLE}. Clearly, $\lambda$ charges no singleton, is fully supported and Radon on $\mathbb{R}$. Let $$\ff(x):=\int_0^x\lambda(dy), \ x\in\mathbb{R}.$$ Then $\ff\in\mathbf{S}(\mathbb{R})$ and $\lambda_\ff=\lambda$. Note that $\lambda_\ff^\e=\lambda_1$ and $\lambda_\ff^\t=\lambda_2$. Then \eqref{EQ5LEF} and \eqref{EQ5LSLE} follow from \eqref{EQ5LLS}. To show \eqref{EQ5LTF}, we take $i\not=j$ with $\lambda_2([e_i,e_j])=0$ where $e_i\in I_i, e_j\in I_j$. Then $$[e_i,e_j]\cap (\mathring{F}\cup K)=\emptyset$$ because of \eqref{EQ5LPF}. Hence
$[e_i,e_j]\setminus \left(\bigcup_n I_n\right)\subset [e_i,e_j]\cap F$ is at most countable and this proves \eqref{EQ5LTF}.
 %it follows that
  % $I_i$ and ${I}_j$ are not tightly scale-connected.  % which leads to \eqref{EQ5LTF}. Indeed, if $\mathring{I}_i\nsim \mathring{I}_j$, then $\mathscr{I}^\sim$ contains infinite intervals and thus $\lambda_2|_{\partial F}([e_i,e_j])>0$ by \eqref{EQ5GIE} and \eqref{EQ5LFI}. This contradicts $\lambda_2([e_i,e_j])=0$.
  That completes the proof.
\end{proof}

Write $\mathbf{S}^\e(\mathbb{R}):=\{\ff\in \mathbf{S}(\mathbb{R}): \ff \text{ satisfies \eqref{EQ5LEF}}\}$.
By Lemma~\ref{LM51} and the above theorem, we can conclude the following corollary.

\begin{corollary}
 %If $\ff\in \mathbf{S}^\e(\mathbb{R})$, then $\CC_\ff$ is a special standard core of a D-subspace of $(\EE,\FF)$. Conversely for any D-subspace $(\EE',\FF')$ of $(\EE,\FF)$, there exists a function $\ff\in \mathbf{S}^\e(\mathbb{R})$ such that $\CC_\ff$ is a special standard core of $(\EE',\FF')$.
$(\EE',\FF')$ is a D-subspace of $(\EE,\FF)$ if and only if there exists a function $\ff\in \mathbf{S}^\e(\mathbb{R})$
such that $\CC_\ff$ is a special standard core of $(\EE',\FF')$. Furthermore, for $\ff_1,\ff_2\in \mathbf{S}^\e(\mathbb{R})$, if $d\ff_1\simeq d\ff_2$, then
$\CC_{\ff_1}$ and $\CC_{\ff_2}$ generate the same D-subspace.
\end{corollary}

The last assertion follows from Corollary \ref{COR57}.

\section{Further remarks}\label{SEC6}

In this section, we shall give several remarks for all the results above. The first remark concerns the state space of 1-dim symmetric diffusions. As mentioned in Remark~\ref{RM22}, all the results can be extended to the 1-dim symmetric diffusions on an interval (not only on $\mathbb{R}$). Furthermore, we shall explain that the presence of killing inside could not affect the discussions about D-subspaces or D-extensions.

\subsection{Diffusions on an interval}

When the state space $\mathbb{R}$ is replaced by an interval $I=\langle l,r \rangle$, the characterization of Dirichlet form associated with an $m$-symmetric diffusion on $I$ is presented in \cite[Theorem~2.1]{LY17}, where $m$ is a fully supported Radon measure on $I$. It is also composed of at most countable disjoint effective intervals $\{I_n \subset I: n\geq 1\}$ and an adapted scale function $\bs_n$ on $I_n$, but now the adapted condition between $I_n$ and $\bs_n$ is a little different. The changes focus on the boundary of $I_n:=\langle a_n, b_n \rangle$, and we say $\bs_n$ is adapted to $I_n$ if
\begin{itemize}
\item[\textbf{(A)}] When $a_n>l$ or $a_n=l\in I$, $a_n\in I_n$ if and only if $\bs_n(a_n)>-\infty$;
\item[\textbf{(B)}] When $b_n<r$ or $b_n=r\in I$, $b_n\in I_n$ if and only if $\bs_n(b_n)<\infty$.
\end{itemize}
Moreover, the Dirichlet form of a diffusion on $I_n$ with the scale function $\bs_n$ is given by \eqref{EQ2FSU2} with replacing $(\text{L}_R)$ and $(\text{R}_R)$ by
\begin{itemize}
\item[(L)] $a_n=l\notin I, \bs_n(a_n)>-\infty$ and $m(a_n+)<\infty$;
\item[(R)] $b_n=r\notin I, \bs_n(b_n)<\infty$ and $m(b_n-)<\infty$.
\end{itemize}
Note that $m(a_n+)<\infty$ (resp. $m(b_n-)<\infty$) means for any $\epsilon >0$,
\[
	m\left((a_n, a_n+\epsilon) \right)<\infty,\quad  (\text{resp. } m\left((b_n-\epsilon, b_n) \right)<\infty).
\]
Loosely speaking, the closed endpoint of $I_n$ must be a reflecting boundary, and the open endpoint of $I_n$, except for $a_n=l\notin I$ or $b_n=r\notin I$, is an unapproachable boundary. Only when $a_n=l\notin I$ or $b_n=r\notin I$, i.e. $a_n$ or $b_n$ shares the same open endpoint of the state space $I$, the boundary $a_n$ or $b_n$ is possibly absorbing. Notice that this is essentially the same as the case $I=\mathbb{R}$, since $a_n=l\notin I$ or $b_n=r\notin I$ corresponds to $a_n=-\infty$ or $b_n=\infty$ when $I=\mathbb{R}$.

All the results in the previous sections can be extended to the cases on the state space $I$. Nothing else need to be changed except for the adapted condition. This conclusion may be easily deduced by mimicking the previous discussions.
For example, the extended result of Theorem~\ref{thm2.1} is given as follows. Note that the effective interval $(I_n, \bs_n)$ enjoys the new adapted condition \textbf{(A)} and \textbf{(B)},
and now the scale measure is a $\sigma$-finite measure on $I$.

\begin{theorem}\label{THM61}
Let $(\mathscr{E},\mathscr{F})$ and $(\mathfrak{E},\mathfrak{F})$ be two regular and strongly local Dirichlet forms on $L^2(I, m)$, whose effective intervals are  $\{(I_n, \bs_n): n\geq 1\}$ and $\{(\mathtt{I}_k, \mathfrak{s}_k): k\geq 1\}$ respectively. Further let $\lambda_{\bs}$ and $\lambda_\mathfrak{s}$ be their scale measures respectively. Then $(\mathfrak{E},\mathfrak{F})$ is a D-subspace of $(\mathscr{E},\mathscr{F})$ on $L^2(I, m)$ if and only if the following conditions hold:
\begin{itemize}
\item[(1)] $\{\mathtt{I}_k: k\geq 1\}$ is coarser than $\{I_n: n\geq 1\}$ in the sense that for any $n$, $I_n\subset \mathtt{I}_k$ for some $k$.
\item[(2)] $\lambda_\mathfrak{s}\ll \lambda_\bs$ and
\[
	\frac{d\lambda_\mathfrak{s}}{d\lambda_\bs}=0\text{ or }1,\quad \lambda_\bs\text{-a.e. on }I.
\]
\end{itemize}
\end{theorem}

Next, we highlight that the state space plays an essential role in the definition of D-subspaces or D-extensions. Recall that in this definition, two Dirichlet forms are imposed to be on the same state space with the same symmetric measure. The following example based on Bessel processes shows us an intuitive illustration.

\begin{example}
Roughly speaking, the $d$-Bessel process $(X_t)_{t\geq 0}$ with an integer $d\geq 2$ is an equivalent version of $(|B_t|)_{t\geq 0}$, where $(B_t)_{t\geq 0}$ is a $d$-dimensional Brownian motion. It is usually treated as a diffusion on $[0,\infty)$, but $0$ is a special boundary: Once leaving $0$, $(X_t)_{t\geq 0}$ will never come back.
As stated in \cite[Example~2.12]{LY17}, $(X_t)_{t\geq 0}$ is symmetric with respect to $m(dx):=x^{d-1}dx$ and its associated Dirichlet form $(\EE,\FF)$ on $L^2([0,\infty), m)$ is regular. Moreover, $\{0\}$ is an $\EE$-exceptional set, and $(\EE,\FF)$ has only one effective interval $I_1:=(0,\infty)$ with the scale function $\bs_1$ on $I_1$:
\[
\bs_1(x):=\left\lbrace \begin{aligned}
 &\log x,\quad d=2; \\
 &\frac{x^{2-d}-1}{2-d},\quad d\geq 3.  	
 \end{aligned}
  \right.
\]
Note that $\bs_1(0)=-\infty$.
On the other hand, $(\EE,\FF)$ can be also treated as a regular Dirichlet form on $L^2((0,\infty), m)$, and any set of singleton is of positive $\EE$-capacity.

Let us consider the D-subspaces of $(\EE,\FF)$ on $L^2((0,\infty),m)$ and $L^2([0,\infty),m)$ respectively. For the case on $L^2((0,\infty),m)$, it follows from Theorem~\ref{THM61} that every D-subspace $(\mathfrak{E}, \mathfrak{F})$ has only one effective interval $(0,\infty)$. Moreover, all the D-subspaces are characterized by the following class of scale functions
\[
\mathfrak{S}:=\left\{\mathfrak{s}_1\in \mathbf{S}((0,\infty)): d\mathfrak{s}_1\ll d\bs_1, \frac{d\mathfrak{s}_1}{d\bs_1}=0 \text{ or }1,\ d\bs_1\text{-a.e.} \right\}.
\]
By Lemma~\ref{resc}, $\mathfrak{S}$ admits an element $\mathfrak{s}_1$ such that $\mathfrak{s}_1(0)>-\infty$, whose associated diffusion is absorbing at $0$.

For the case on $L^2([0,\infty),m)$, a D-subspace also has only one effective interval $\mathtt{I}_1$ but admits two possibilities: $\mathtt{I}_1=(0,\infty)$ or $[0,\infty)$. If $\mathtt{I}_1=(0,\infty)$, $(\mathfrak{E}, \mathfrak{F})$ is characterized by a scale function in the subclass of $\mathfrak{S}$:
\[
\mathfrak{S}_\infty:=\{\mathfrak{s}_1\in \mathfrak{S}: \mathfrak{s}_1(0)=-\infty\}.
\]
Clearly, $\{0\}$ is always an $\mathfrak{E}$-exceptional set in this case.
If $\mathtt{I}_1=[0,\infty)$, $(\mathfrak{E},\mathfrak{F})$ is characterized by a scale function in another class:
\[
	\mathfrak{S}_\mathrm{finite}:=\{\mathfrak{s}_1\in \mathfrak{S}: \mathfrak{s}_1(0)>-\infty\}.
\]
Notice that $\mathfrak{S}=\mathfrak{S}_\infty \cup \mathfrak{S}_\mathrm{finite}$. It is worth noting that in this case, $\mathfrak{s}_1\in\mathfrak{S}_\mathrm{finite}$ corresponds to a D-subspace, whose associated diffusion is reflecting at $0$.
\end{example}

\subsection{Killing inside}

We always impose the Dirichlet forms to have no killing inside in previous sections, since killing insides are not essential for the discussions of D-subspaces or D-extensions. Now we shall briefly explain this imposition and present the results with killing insides.

By \cite[Theorem~4.1]{LY17}, a regular and local Dirichlet form $(\EE,\FF)$ on $L^2(I,m)$ is characterized by a class of effective intervals $\{(I_n, \bs_n): n\geq 1\}$ and a killing measure $k$, which is Radon on $I$ and such that $k\ll m$ on $I\setminus \cup_{n\geq 1}I_n$. Let $(\EE^0,\FF^0)$ be given by the same effective intervals but with no killing inside. Due to \cite[Theorem~4.1]{LY17}, we know that $(\EE^0,\FF^0)$ is the resurrected Dirichlet form of $(\EE,\FF)$, and $(\EE,\FF)$ is the perturbed Dirichlet form of $(\EE^0,\FF^0)$ induced by $k$. On the other hand, it follows from \cite[Theorem~2.1]{LY15} that any D-subspace or D-extension of $(\EE,\FF)$ has the same killing measure as $(\EE,\FF)$. Therefore, we can obtain the extended result of Theorem~\ref{THM61} as follows.

\begin{theorem}\label{THM63}
Let $(\EE,\FF)$ be a regular and local Dirichlet form on $L^2(I,m)$ with effective intervals $\{(I_n, \bs_n): n\geq 1\}$ and killing measure $k$. Further let $(\mathfrak{E}, \mathfrak{F})$ be another regular and local Dirichlet form on $L^2(I,m)$ with effective intervals $\{(\mathtt{I}_k, \mathfrak{s}_k): k\geq 1\}$ and killing measure $\mathfrak{k}$. Then $(\mathfrak{E},\mathfrak{F})$ is a D-subspace of $(\mathscr{E},\mathscr{F})$ on $L^2(I, m)$ if and only if the following conditions hold:
\begin{itemize}
\item[(1)] $\{\mathtt{I}_k: k\geq 1\}$ is coarser than $\{I_n: n\geq 1\}$ in the sense that for any $n$, $I_n\subset \mathtt{I}_k$ for some $k$.
\item[(2)] $\lambda_\mathfrak{s}\ll \lambda_\bs$ and
\[
	\frac{d\lambda_\mathfrak{s}}{d\lambda_\bs}=0\text{ or }1,\quad \lambda_\bs\text{-a.e. on }I.
\]
\item[(3)] $\mathfrak{k}=k$.
\end{itemize}
\end{theorem}
\begin{proof}
Let $(\mathfrak{E}^0, \mathfrak{F}^0)$ be the resurrected Dirichlet form of $(\mathfrak{E}, \mathfrak{F})$. In other words, $(\mathfrak{E}^0,\mathfrak{F}^0)$ is a Dirichlet form with the effective intervals $\{(\mathtt{I}_k,\mathfrak{s}_k): k\geq 1\}$. Then we only need to point out that $(\mathfrak{E}, \mathfrak{F})$ is a D-subspace of $(\EE,\FF)$ if and only if $(\mathfrak{E}^0,\mathfrak{F}^0)$ is a D-subspace of $(\EE^0,\FF^0)$ and $k=\mathfrak{k}$. That completes the proof.
\end{proof}

\begin{remark}
That $\mathfrak{k}$ is the killing measure of $(\mathfrak{E},\mathfrak{F})$ amounts to $\mathfrak{k}\ll m$ on $I\setminus \cup_{k\geq 1}\mathtt{I}_k$. The third condition in Theorem~\ref{THM63} indicates that when $(\mathfrak{E},\mathfrak{F})$ is a D-subspace of $(\EE,\FF)$, it has more imposition: $\mathfrak{k}\ll m$ on $I\setminus \cup_{n\geq 1}I_n$ (Note that $\cup_{n\geq 1}I_n\subset \cup_{k\geq 1}\mathtt{I}_k$).
\end{remark}

Roughly speaking, to study a problem about D-subspaces (or D-extensions) of a Dirichlet form with killing insides, we could first consider the analogical problem of the resurrected Dirichlet form and then come back to it by killing transform. Particularly, nothing else need to be changed to derive the analogical results of previous sections for Dirichlet forms with killing insides, except for an additional condition: $\mathfrak{k}=k$.

 %\section*{Acknowledgement}

 %\bibliographystyle{amsplain}
 %\bibliography{subspace}

 %\bibliographystyle{amsplain}

\end{document}